\documentclass[11pt]{amsart}
\usepackage{amssymb,amsmath,amsfonts,amsthm, fullpage,mathrsfs}
\usepackage{float}
\numberwithin{equation}{section}

\usepackage{color}

\newcommand{\gone}[1]{{\color{yellow}}}

\usepackage{verbatim,moreverb}

\newcommand{\FF}{\mathbb F}
\newcommand{\GG}{\mathbb G}

\newcommand{\QQ}{\mathbb Q}

\newcommand{\ZZ}{\mathbb Z} 

\newcommand{\Zhat}{\widehat\ZZ}

\newcommand{\C}{\mathcal C} 

\newcommand{\OO}{\mathcal O}

\newcommand{\p}{\mathfrak p} \newcommand{\Pp}{\mathfrak P}

 \def\Gal{\operatorname{Gal}}

\def \GL {\operatorname{GL}_2}  

\def \SL {\operatorname{SL}_2}

\def\Aut{\operatorname{Aut}} \def\End{\operatorname{End}}

 \def\Fr{\operatorname{Frob}}
\def\lcm{\operatorname{lcm}}

\def \B {\mathcal B}
\def\calF{\mathcal{F}}

\def\tors{\operatorname{tors}}

\def\bbar#1{\setbox0=\hbox{$#1$}\dimen0=.2\ht0 \kern\dimen0 \overline{\kern-\dimen0 #1}}
\newcommand{\Qbar}{{\overline{\mathbb Q}}} 
\newcommand{\Kbar}{\bbar{K}}

\newtheorem{thm}{Theorem}[section]
\newtheorem{lemma}[thm]{Lemma}

\newtheorem{prop}[thm]{Proposition}
\newtheorem{conj}[thm]{Conjecture}

\theoremstyle{definition}
\newtheorem{definition}[thm]{Definition}

\theoremstyle{remark}
\newtheorem{remark}[thm]{Remark}

\newenvironment{romanenum}{\hfill \begin{enumerate} }{\end{enumerate}}

\definecolor{webbrown}{rgb}{.6,0,0}
\usepackage[
        colorlinks,
        linkcolor=webbrown,  filecolor=webcolor,  citecolor=webbrown, 
        backref,
        pdfauthor={David Zywina}, % add other authors
]{hyperref}
\usepackage[alphabetic,backrefs,lite]{amsrefs} % for bibliography
\begin{document}

\title[A refinement of Koblitz's conjecture]{A refinement of Koblitz's conjecture}
%Elliptic curves over global fields, distribution of primes

\author{David Zywina}
\email{zywina@math.upenn.edu}
\address{Department of Mathematics, University of Pennsylvania, Philadelphia, PA 19104-6395, USA}
\date{\today}

\begin{abstract} 
Let $E$ be an elliptic curve over the number field $\QQ$.  In 1988, Koblitz conjectured an asymptotic for the number of primes $p$ for  which the cardinality of the group of $\FF_p$-points of $E$ is prime.   However, the constant occurring in his asymptotic does not take into account that the distributions of the $|E(\FF_p)|$ need not be independent modulo distinct primes.  We shall describe a corrected constant.  We also take the opportunity to extend the scope of the original conjecture to ask how often $|E(\FF_p)|/t$ is prime for a fixed positive integer $t$, and to consider elliptic curves over arbitrary number fields.   Several worked out examples are provided to supply numerical evidence for the new conjecture. 
\end{abstract}

\maketitle

%**************************************************************************** 
\section{Introduction} \label{S:intro}

Motivated by applications to elliptic curve cryptography and the heuristic methods of Hardy and Littlewood \cite{Hardy-Littlewood}, N. Koblitz made the following conjecture: 
\begin{conj}[\cite{Koblitz}*{Conjecture A}]  \label{C:original}
Let $E$ be a non-CM elliptic curve defined over $\QQ$ with conductor $N_E$.  Assume that $E$ is not $\QQ$-isogenous to a curve with nontrivial $\QQ$-torsion.  Then 
\[
|\{ p\leq x \text{ prime}: p\nmid N_E,  \,|E(\FF_p)| \text{ is prime}\}| \sim C_E \frac{x}{(\log x)^2}
\]
as $x\to \infty$, where $C_E$ is an explicit positive constant.  
\end{conj}

However, the description of the constant $C_E$ in \cite{Koblitz} is not always correct (and more seriously, our corrected version of the constant is not necessarily positive).  The additional phenomena that needs to be taken into account is that the divisibility conditions modulo distinct primes, unlike the more classical cases considered by Hardy and Littlewood, need not be independent.  Lang and Trotter have successfully dealt with this non-independence in their conjectures \cite{Lang-Trotter}.     A similar modification was required for the original constant of Artin's conjecture; see \cite{Stevenhagen} for a nice historical overview. 

\subsection{An example} \label{SS:an example}
As an illustration, consider the following example kindly provided by N. Jones.  Let $E$ be the elliptic curve over $\QQ$ defined by the Weierstrass equation $y^2 = x^3 + 9x+18$; this curve has conductor $2^4 3^4$, and is not isogenous over $\QQ$ to an elliptic curve with non-trivial $\QQ$-torsion.  Conjecture~\ref{C:original} predicts that $|E(\FF_p)|$ is prime for infinitely many primes $p$; however for $p>5$, $|E(\FF_p)|$ is always composite!   

For a positive integer $m$, let $\vartheta_m$ be the density of the set of primes $p$ for which $|E(\FF_p)|$ is divisible by $m$; intuitively, we may think of this as the probability that $m$ divides $|E(\FF_p)|$ for a ``random'' $p$.  We can compute these $\vartheta_m$ by applying the Chebotarev density theorem to the extensions $\QQ(E[m])/\QQ$, where $\QQ(E[m])$ is the extension of $\QQ$ generated by the coordinates of the $m$-torsion points of $E$.   For our elliptic curve, we have $\vartheta_2=2/3$ and $\vartheta_3=3/4$.   It is thus natural to expect that $\vartheta_6 = \vartheta_2 \vartheta_3 = 1/2$ (i.e., that the congruences modulo $2$ and $3$ are \emph{independent} of each other); however, one actually has $\vartheta_6 = 5/12$.  The inclusion-exclusion principle then tells us that the ``probability'' that $|E(\FF_p)|$ is relatively prime to $6$ is $1-\vartheta_2-\vartheta_3 + \vartheta_6=0$.

This lack of independence is explained by the observation that $\QQ(E[2])$ and $\QQ(E[3])$ are not linearly disjoint over $\QQ$.  They both contain $\QQ(i)$:
\begin{itemize}
\item 
The point $(x,y)=(-3,6i)$ in $E(\QQ(i))$ has order $3$, so $\QQ(E[3])$ contains $\QQ(i)$.  If $p$ splits in $\QQ(i)$ (i.e., $p\equiv 1\bmod{4}$), then $(-3,6i)$ will give a point in $E(\FF_p)$ of order 3; hence $|E(\FF_p)|\equiv 0 \bmod{3}$.
\item
The points in $E[2]-\{0\}$ are of the form $(x,0)$, where $x$ is a root of $x^3+9x+18$. The discriminant of this cubic is $\Delta=-2^4 3^6$, so $\QQ(E[2])$ contains $\QQ(\sqrt{\Delta})=\QQ(i)$.  If $p>3$ is inert in $\QQ(i)$ (i.e., $p\equiv 3\bmod{4}$), then $\Delta$ is not a square modulo $p$ and one checks that $E(\FF_p)$ has exactly one point of order $2$; hence $|E(\FF_p)|\equiv 0 \bmod{2}$.
\end{itemize}
For $p\geq 5$, we deduce that $|E(\FF_p)|$ is divisible by $2$ or $3$.  Therefore $|E(\FF_p)|$ is prime only in the case where it equals $2$ or $3$ (which happens for $p=5$ when $|E(\FF_5)|=3$).  

It is now natural to ask if $|E(\FF_p)|/3$ (or $|E(\FF_p)|/2$) is prime for infinitely many $p$?  Our refinement/generalization of Koblitz's conjecture predicts that the answer is \emph{yes}, and we will supply numerical evidence in \S\ref{S:Jones example}.

%*************************************************************
\subsection{The refined Koblitz conjecture} \label{SS:conj}

Before stating our conjecture, we set some notation that will hold throughout the paper.   For a number field $K$, denote the ring of integers of $K$ by $\OO_K$, and let $\Sigma_K$ be the set of non-zero prime ideals of $\OO_K$.  For each prime $\p\in \Sigma_K$, we have a residue field $\FF_\p=\OO_K/\p$ whose cardinality we denote by $N(\p)$.    Let $\Sigma_K(x)$  be the (finite) set of primes $\p\in\Sigma_K$ with $N(\p)\leq x$. 

For an elliptic curve $E$ over $K$, let $S_E$ be the set of $\p\in \Sigma_K$ for which $E$ has bad reduction.  For $\p\in\Sigma_K-S_E$, let $E(\FF_\p)$ be the corresponding group of $\FF_\p$-points (more precisely, the $\FF_\p$-points of the N\'eron model $\mathbb{E}/\OO_K$ over $E/K$).  For a field extension $L/K$, we will denote by $E_L$ the corresponding base extension of $E$.

\begin{conj} \label{C:main}
Let $E$ be an elliptic curve defined over a number field $K$, and let $t$ be a positive integer. Then there is an explicit constant $\C_{E,t}\geq 0$ such that
\[
P_{E,t}(x):=|\{\p \in \Sigma_K(x)-S_E :  |E(\FF_\p)|/t \text{ is a prime} \}| \sim \C_{E,t} \frac{x}{(\log x)^2}
\]
as $x\to\infty$. 
\end{conj}

If $\C_{E,t}=0$, then we \emph{define} the above asymptotic to mean that $P_{E,t}(x)$ is bounded as a function of $x$ (equivalently, that $|E(\FF_\p)|/t$ is prime for only finitely many $\p\in \Sigma_K-S_E$). Our constant $\C_{E,t}$ will be described in \S\ref{S:main section}.

The expression $\C_{E,t} \, x /(\log x)^2$ in Conjecture~\ref{C:main} has been used for its simplicity. The heuristics in \S\ref{SS:heuristics} suggest that the expression
\begin{equation} \label{E:best expression}
\C_{E,t} \,\int^x_{t+1} \frac{1}{ \log (u+1) -\log t} \frac{du}{\log u}
\end{equation}
will be a better approximation of $P_{E,t}(x)$, and this is what we will use to test our conjecture.    We will not study the error term of our conjecture (i.e., the difference between $P_{E,t}(x)$ and the expression (\ref{E:best expression})), though we remark that our data suggests that it could be $O(x^\theta)$ for any $\theta>1/2$.

%*************************************************************
\subsection{Overview}
In \S\ref{S:main section}, we describe the constant $\C_{E,t}$ occurring in Conjecture~\ref{C:main}.   We shall express the constant in terms of the Galois representations arising from the torsion points of our elliptic curve.  To have a computationally useful version, we treat separately the CM and non-CM cases.   In \S\ref{SS:heuristics} we give a brief heuristic for our conjecture.
In \S\ref{S:common factor}, we describe the common factor $t_E$ of all the $|E(\FF_\p)|$.  It is of course necessary to have $t_E$ divide $t$ for Conjecture~\ref{C:main} to be interesting.  In \S\ref{S:Serre curve section}, we calculate $\C_{E,1}$ assuming that $E/\QQ$ is a Serre curve.
In \S\ref{S:main example}--\ref{S:X0(11)}, we consider four specific elliptic curves.   We describe the Galois action on their torsion points,  compute constants $\C_{E,t}$ for interesting $t$, and then supply  numerical evidence for Conjecture~\ref{C:main}.
In the final section, we describe some of the partial progress that has been made on Koblitz's conjecture in the last decade.

\subsection*{Acknowledgments}
Thanks to Nathan Jones for comments and providing the example in \S\ref{SS:an example}.  Special thanks to Chantal David and Bjorn Poonen.  The experimental evidence for our conjecture was computed using \texttt{PARI/GP} \cite{PARI}.  We also used \texttt{Magma} \cite{Magma} to check some group theoretic claims and \texttt{Maple} to approximate integrals.  This research was supported by an NSERC postgraduate scholarship.

%*************************************************************
\section{The constant} \label{S:main section}

Throughout this section, we will fix an elliptic curve $E$ defined over a number field $K$ and a positive integer $t$.  The letter $\ell$ will always denote a rational prime.

\subsection{Description of the constant}
To understand the divisibility of the numbers $|E(\FF_\p)|$, it is useful to recast everything in term of Galois representations.  For each positive integer $m$, let $E[m]$ be the group of $m$-torsion in
$E(\Kbar)$, where $\Kbar$ is a fixed algebraic closure of $K$.  The natural Galois action induces a representation
\[
\rho_m \colon \Gal(\Kbar/K) \to \Aut(E[m])
\]
whose image we will denote by $G(m)$.  Let $K(E[m])$ be the fixed field of $\ker(\rho_m)$ in $\Kbar$; so $\rho_m$ induces an isomorphism $\Gal(K(E[m])/K) \xrightarrow{\sim} G(m)$. If $\p\in \Sigma_K-S_E$ does not divide $m$, then $\rho_m$ is unramified at $\p$ (i.e., $\p$ is unramified in $K(E[m])$) and $\rho_m(\Fr_\p)$ will denote the corresponding Frobenius conjugacy class in $G(m)$.   Note that the notation does not mention the curve $E$ which will always be clear from context.

The group $E[m]$ is a free $\ZZ/m\ZZ$-module of rank $2$; a choice of $\ZZ/m\ZZ$-basis for $E[m]$ determines an isomorphism $\Aut(E[m])\cong \GL(\ZZ/m\ZZ)$ that is unique up to an inner automorphism of $\GL(\ZZ/m\ZZ)$.   For a prime ideal $\p\in\Sigma_K-S_E$ with $\p\nmid m$, we have a congruence
\begin{equation*}
|E(\FF_\p)| \equiv \det( I - \rho_m(\Fr_\p))\; \bmod{m}.
\end{equation*}

For $m\geq 1$, define the set
 \begin{equation} \label{E:Psi}
 \Psi_t(m) = \big\{ A \in \Aut(E[m]) : \det(I-A) \in t\cdot (\ZZ/m\ZZ)^\times \big\}.
 \end{equation}
Thus for a prime $\p\in \Sigma_K-S_E$ with $\p \nmid m$, we have
\begin{equation}\label{E:rep connection}
 |E(\FF_\p)|/t \text{ is invertible modulo } \frac{m}{\gcd(m,t)}\quad \text{ if and only if } \quad  \rho_m(\Fr_\p) \subseteq G(m)\cap \Psi_t(m).
\end{equation}
In particular, $|E(\FF_\p)|/t$ is an integer if and only if $\rho_t(\Fr_\p) \subseteq G(t)\cap \Psi_t(t)$.  Define the number
\[
 \delta_{E,t}(m):=\frac{|G(m)\cap \Psi_t(m)|}{|G(m)|}
\]
By (\ref{E:rep connection}) and the Chebotarev density theorem, $\delta_{E,t}(m)$ is the natural density of the set of $\p\in \Sigma_K-S_E$ for which $|E(\FF_\p)|/t$ is invertible modulo $m/\gcd(m,t).$  The connection with Conjecture~\ref{C:main} is that if $|E(\FF_\p)|/t$ is a prime number, then it is invertible modulo all
integers $m<|E(\FF_\p)|/t$.

\begin{definition} \label{D:C defn}
With notation as above, define
\[
 \C_{E,t} := \lim_{m\to +\infty} \frac{\delta_{E,t}(m)}{\prod_{\ell| m}
(1-1/\ell)}
\]
where the limit runs over all positive integers ordered by divisibility; this is our predicted constant for Conjecture~\ref{C:main}.  An equivalent definition is
\[
\C_{E,t} = \lim_{Q\to +\infty} \frac{\delta_{E,t}\big(t\prod_{\ell\leq
Q}\ell\big)}{\prod_{\ell\leq Q} (1-1/\ell)}
\]
since for $m$ divisible by $t\prod_{\ell|t}\ell$, we have $\delta_{E,t}(m) = \delta_{E,t}(t{\prod}_{\ell|m} \ell)$.
\end{definition}

We shall see in \S\ref{S:nonCM constant} and \S\ref{S:CM constant}, that the limits of Definition~\ref{D:C defn} do indeed converge, and hence $\C_{E,t}$ is well-defined.  It will also be apparent that $\C_{E,t}=0$ if and only if $\delta_{E,t}(m)=0$ for some $m$; this gives the following qualitative version of our conjecture:

\begin{conj}
Let $E$ be an elliptic curve over a number field $K$, and let $t$ be a positive integer.   There are infinitely many $\p\in \Sigma_K$ for which $|E(\FF_\p)|/t$ is prime if and only if there are no ``congruence obstructions'', i.e., for every $m\geq 1$ there exists a prime $\p\in \Sigma_K-S_E$ with $\p\nmid m$ such that $|E(\FF_\p)|/t$ is invertible modulo $m$.
\end{conj}

%*************************************************************
\subsection{The constant for non-CM elliptic curves} \label{S:nonCM constant}
The following renowned theorem of Serre, gives the general structure of the groups $G(m)$.

\begin{thm}[Serre \cite{SerreInv}]  \label{T:openimage}
Let $E/K$ be an elliptic curve without complex multiplication.  There is a positive integer $M$ such that if $m$ and $n$ are positive integers with $n$ relatively prime to $Mm$, then
\[
 G(mn) = G(m) \times \Aut(E[n]).
\]
\end{thm}

\begin{prop} \label{P:how to}
Let $E/K$ be an elliptic curve without complex multiplication, and let $t$ be a positive integer.  Let $M$ be a positive integer such that
\[
 G\Big(t \prod_{\ell| tm}\ell\Big) = G\Big(t \prod_{\ell|t\gcd(M,m)}
\ell\Big) \times \prod_{\ell|m,\, \ell\nmid tM} \Aut(E[\ell])
\]
for all (squarefree) $m$ (in particular, one can take $M$ as in Theorem~\ref{T:openimage}). Then
\[
 \C_{E,t} = \dfrac{ \delta_{E,t}\big(t\prod_{\ell|tM}\ell\big)}{\prod_{\ell|tM} (1-1/\ell)} \prod_{\ell \nmid tM} \Big(1 - \frac{\ell^2-\ell-1}{(\ell-1)^3(\ell+1)}\Big).
\]
\end{prop}
\begin{proof}
Let $Q$ be a real number greater than $tM$.  From the assumption of the proposition, we have
\[
G(t{\prod}_{\ell\leq Q}\ell) = G(t{\prod}_{\ell|tM}\ell)\times \prod_{\ell\nmid tM, \ell\leq Q} \Aut(E[\ell]).
\]
Therefore
\begin{align*}
\delta_{E,t}\big(t{\prod}_{\ell\leq Q}\ell\big) &= \delta_{E,t}\big(t{\prod}_{\ell|tM}\ell\big) \prod_{\ell\nmid tM, \ell\leq Q} \delta_{E,t}(\ell),
\end{align*}
and hence
\begin{align} \label{E:indepency}
\frac{\delta_{E,t}\big(t{\prod}_{\ell\leq Q}\ell\big)}{{\prod}_{\ell\leq Q}(1-1/\ell)} &= \frac{\delta_{E,t}\big(t{\prod}_{\ell|tM}\ell\big)}{{\prod}_{\ell|tM}(1-1/\ell)} \prod_{\ell\nmid tM, \ell\leq Q} \frac{\delta_{E,t}(\ell)}{1-1/\ell}.
\end{align}
For any $\ell \nmid tM$,  we have
\begin{align*}
\delta_{E,t}(\ell) &= 1 - \frac{|\{A \in \GL(\FF_\ell) : \det(I-A)= 0\}|}{|\GL(\FF_\ell)|}\\ 
& = 1 - \sum_{a\in \FF_\ell^\times} \frac{|\{ A\in \GL(\FF_\ell): \text{the eigenvalues of $A$ are $1$ and $a$}\}|}{|\GL(\FF_\ell)|}
\end{align*}
and by Lemma~\ref{L:eigenvalues} below,
\[
\frac{\delta_{E,t}(\ell)}{1-1/\ell}  = \frac{1}{1-1/\ell} \Big(1 - \frac{(\ell-2)(\ell^2+\ell)   + 1\cdot \ell^2 }{{\ell(\ell-1)^2(\ell+1)}}\Big).
\]
A easy calculation then shows that $\dfrac{\delta_{E,t}(\ell)}{1-1/\ell} = 1 - \dfrac{\ell^2-\ell-1}{(\ell-1)^3(\ell+1)}$.  Substituting this into (\ref{E:indepency}), gives
 \[
\frac{\delta_{E,t}\big(t{\prod}_{\ell|tM}\ell\big)}{{\prod}_{\ell|tM}(1-1/\ell)} \prod_{\ell\nmid tM, \ell\leq Q} \Big(1 - \frac{\ell^2-\ell-1}{(\ell-1)^3(\ell+1)}\Big).
\]
Letting $Q\to +\infty$, we deduce that the limit defining $\C_{E,t}$ is convergent and that it has the stated value.
\end{proof}

\begin{lemma} \label{L:eigenvalues}
For $a\in \FF_\ell^\times$,
\[
|\{ A\in \GL(\FF_\ell): \text{the eigenvalues of $A$ are $1$ and $a$}\}| = \begin{cases}
     \ell^2+\ell & \text{if }a\neq 1, \\
     \ell^2 & \text{if } a=1.
\end{cases}
\]
\end{lemma}
\begin{proof}
This follows easily from Table~12.4 in \cite{Lang Algebra}*{XVIII}, which describes the conjugacy classes of $\GL(\FF_\ell)$.
\end{proof}

\begin{remark} 
For later reference, we record the following numerical approximation:
\begin{equation} \label{E:numerical constant}
\mathfrak{C} :=\prod_\ell \Bigl(1 - \frac{\ell^2-\ell-1}{(\ell-1)^3(\ell+1)}\Bigr)\approx 0.505166168239435774.
\end{equation}
So to estimate $\C_{E,t}$, it suffices to find $M$ and then compute $\delta_{E,t}(t\prod_{\ell|tM}\ell)$.
\end{remark}

%*************************************************************
\subsection{The constant for CM elliptic curves} \label{S:CM constant}
Let $E$ be an elliptic curve over a number field $K$ with complex multiplication, and let $R=\End(E_{\Kbar})$.   The ring $R$ is an order in the imaginary quadratic field $F:=R\otimes_\ZZ \QQ$.

For each positive integer $m$, we have a natural action of $R/mR$ on $E[m]$.  The group $E[m]$ is a free $R/mR$-module of rank $1$, so we have a \emph{canonical} isomorphism $\Aut_{R/mR}(E[m]) = (R/mR)^\times$.  If all the endomorphism of $E$ are defined over $K$, then the actions of $R$ and $\Gal(\Kbar/K)$ on $E[m]$ commute, and
hence we may view $\rho_m(\Gal(\Kbar/K))$ as a subgroup of $(R/mR)^\times$.

\begin{prop} \label{P:CMopenimage}
 Let $E$ be an elliptic curve over a number field $K$ with complex multiplication.  Assume that all the endomorphisms in $R=\End(E_{\Kbar})$ are defined over $K$.  There is a positive integer $M$ such that if $m$ and $n$ are positive integers with $n$ relatively prime to $Mn$, then
\[
 G(mn)= G(m) \times (R/n R)^\times.
\]
\end{prop}
\begin{proof}
(For an overview and further references, see \cite{SerreInv}*{\S4.5}) For a prime $\ell$, define $R_\ell=R\otimes_\ZZ \ZZ_\ell$ and $F_\ell=F\otimes_\QQ \QQ_\ell$.  Let $T_\ell(E)$ be the $\ell$-adic Tate module of $E$ (i.e, the inverse limit of the groups $E[\ell^i]$ with multiplication by $\ell$ as transition maps).  The Tate module
$T_\ell(E)$ is a free $R_\ell$-module of rank $1$ (see the remarks at the end of \S4 of \cite{SerreTate:GoodReduction}); we thus have a \emph{canonical} isomorphism $\Aut_{R_\ell}(T_\ell (E)) = R_\ell^\times$.  The actions of $\Gal(\Kbar/K)$ and $R_\ell$ on $T_\ell(E)$ commute with each other (since we have assumed that all
the endomorphisms of $E$ are defined over $K$).  Combining our representations $\rho_{\ell^i}$ gives a Galois representation
\[
\widehat\rho_\ell \colon \Gal(\Kbar/K) \to \Aut_{R_\ell}(T_\ell (E)) = R_\ell^\times.
\]
The theory of complex multiplication implies that the representation
\[
\widehat{\rho}:=\prod_{\ell} \widehat\rho_\ell \colon \Gal(\Kbar/K) \to \prod_{\ell} R_\ell^\times
\]
has open image; our proposition is an immediate consequence.\\

We now describe the representation $\widehat{\rho}$ in further detail (this will be useful later when we actually want to compute a suitable $M$).  Since the endomorphism in $R$ are defined over $K$, the action of $R$ on the Lie algebra of $E$ gives a homomorphism $R\to K$.  This allows us to identify $F$ with a subfield of $K$.  By class field theory, we may
view $\widehat\rho_\ell$ as a continuous homomorphism $I \to R_\ell^\times\subseteq F_\ell^\times$ that is trivial on
$K^\times$, where $I$ is the group of ideles of $K$ with its standard topology.  For each prime $\ell$, define $K_\ell := K \otimes_\ZZ \QQ_\ell = \prod_{\p | \ell} K_\p$.  For an element $a\in I$, let $a_\ell$ be the component of $a$ in $K_\ell^\times$.   From \cite{SerreTate:GoodReduction}*{\S4.5 Theorems 10 \& 11}, there is a unique homomorphism $\varepsilon\colon I \to F^\times$ such that
\[
\widehat\rho_\ell(a) = \varepsilon(a) N_{K_\ell/F_\ell}(a_\ell^{-1})
\]
for all $\ell$ and $a\in I$.   The homomorphism $\varepsilon$ is continuous and  $\varepsilon(x)=x$ for all $x\in K^\times$.

Since $\varepsilon$ is continuous, there is a set $S\subseteq \Sigma_K$ such that $\varepsilon$ is $1$ on $\prod_{\p\in\Sigma_K -S}\OO_{K,\p}^\times\subseteq I$; in fact, we may take $S=S_E$.  Let $M$ be a positive integer such that
\begin{itemize}
 \item $N_{K_\ell/F_\ell}\colon (\OO_K\otimes \ZZ_\ell)^\times \to R_\ell^\times$ is surjective for all $\ell\nmid M$.
\item
$E$ has good reduction at all $\p\in \Sigma_K$ for which $\p\nmid M$.
\end{itemize}
Take any $b= (b_\ell) \in \prod_\ell R_\ell^\times$ with $b_\ell=1$ for all $\ell|M$.  For each $\ell$, there is an $a_\ell \in(\OO_K\otimes \ZZ_\ell)^\times \subseteq K_\ell^\times$ such that $N_{K_\ell/F_\ell}(a_\ell^{-1})=b_\ell$.  Let $a$ be the corresponding element of $I$ with archimedean component equal to $1$.  Then
\[
\widehat{\rho}(a) = (\widehat\rho_\ell(a))_\ell = (\varepsilon(a) N_{K_\ell/F_\ell}(a_\ell^{-1}))_\ell = (N_{K_\ell/F_\ell}(a_\ell^{-1}))_\ell = (b_\ell)_\ell.
\]
Since $(b_\ell)$ was an arbitrary element of $\prod_{\ell}R_\ell^\times$ with $b_\ell=1$ for $\ell|M$, we conclude
that $\widehat{\rho}(\Gal(\Kbar/K))\supseteq \{1\}\times \prod_{\ell\nmid M} R_\ell^\times$.  Our $M$ thus agrees with the one in the statement of the propostion.
\end{proof}

\begin{prop} \label{P:CM how to}
Let $E$ be an elliptic curve over a number field $K$ with complex multiplication.  Assume that all the endomorphisms in $R=\End(E_{\Kbar})$ are defined over $K$.  Let $\chi$ be the Kronecker character corresponding to the imaginary quadratic extension $F=R\otimes\QQ$ of $\QQ$.  Let $M$ be a positive integer as in Proposition~\ref{P:CMopenimage} which is also divisible by all the primes dividing the discriminant of $F$ or the conductor of the order
$R$.  For any positive integer $t$, we have
\[
 \C_{E,t} = \dfrac{ \delta_{E,t}\big(t\prod_{\ell|tM}\ell\big) }{\prod_{\ell|tM} (1-1/\ell)} \cdot \prod_{\ell \nmid tM} \Big(1 - \chi(\ell) \frac{\ell^2-\ell-1}{(\ell-\chi(\ell))(\ell-1)^2}\Big).
\]
\end{prop}
\begin{proof}
Let $Q$ be a real number greater than $tM$.  By Proposition~\ref{P:CMopenimage}, we have
\[
 G(t{\prod}_{\ell\leq Q} \ell) =  G(t{\prod}_{\ell|tM}\ell)\times \prod_{\ell\nmid tM, \ell\leq Q} (R/\ell R)^\times.
\]
Therefore
\begin{align*}
\delta_{E,t}\big(t{\prod}_{\ell\leq Q}\ell\big) &= \delta_{E,t}\big(t{\prod}_{\ell|tM}\ell\big) \prod_{\ell\nmid tM,
\ell\leq Q} \delta_{E,t}(\ell),
\end{align*}
and hence
\begin{align} \label{E:CMindepency}
\frac{\delta_{E,t}\big(t{\prod}_{\ell\leq Q}\ell\big)}{{\prod}_{\ell\leq Q}(1-1/\ell)} &= \frac{\delta_{E,t}\big(t{\prod}_{\ell|tM}\ell\big)}{{\prod}_{\ell|tM}(1-1/\ell)} \prod_{\ell\nmid tM, \ell\leq Q} \frac{\delta_{E,t}(\ell)}{1-1/\ell}.
\end{align}
Now take any $\ell \nmid tM$.  Under the identification $\Aut_{R/\ell R}(E[\ell])=(R/\ell R)^\times$, for $a\in (R\ell R)^\times$ we find that $\det(I-a)$ agrees with $N(1-a)$ where $N$ is the norm map from $R/\ell R$ to $\ZZ/\ell\ZZ$.  We then have
\begin{align*}
\delta_{E,t}(\ell) &= \frac{|\{a \in (R/\ell R)^\times : N(1- a) \in (\ZZ/\ell\ZZ)^\times\}|}{|(R/\ell R)^\times|}
= \frac{|\{a \in (R/\ell R)^\times : 1- a \in (R/\ell R)^\times\}|}{|(R/\ell R)^\times|}.
\end{align*}
From our assumptions on $M$, $\ell$ is unramified in $F$ and $R/\ell R= \OO_F /\ell \OO_F$. One can then verify that $|(\OO_F/\ell\OO_F)^\times|=(\ell-1)(\ell-\chi(\ell)),$ and
\[
|\{ a \in (\OO_F/\ell\OO_F)^\times : 1-a \in(\OO_F/\ell\OO_F)^\times\}| = \ell^2 - \big(\chi(\ell)(\ell-2)+(\ell-1) \big).
\]
An easy calculation then shows 
$\dfrac{\delta_{E,t}(\ell)}{1-1/\ell}=1 - \chi(\ell) \frac{\ell^2-\ell-1}{(\ell-\chi(\ell))(\ell-1)^2}$. Substituting this into (\ref{E:CMindepency}), gives
 \[
\frac{\delta_{E,t}\big(t{\prod}_{\ell|tM}\ell\big)}{{\prod}_{\ell|tM}(1-1/\ell)} \prod_{\ell\nmid tM, \ell\leq Q} \Big(1 - \chi(\ell) \frac{\ell^2-\ell-1}{(\ell-\chi(\ell))(\ell-1)^2}\Big).
\]
Letting $Q\to +\infty$, we deduce that the limit defining $\C_{E,t}$ is (conditionally) convergent and has the stated value (the convergence can be seen by a comparison with the Euler product of the $L$-function $L(s,\chi)$ at $s=1$ which converges to a non-zero number).
\end{proof}

\subsubsection{Case where not all the endomorphisms are defined over base field}
Let's now consider the case where not all the endomorphisms of $E$ over $K$.  Choose an embedding $F\subseteq \Kbar$.  The endomorphisms of $E$ are defined over $KF$, and $KF$ is a quadratic extension of $K$.  We break up the conjecture into two cases.\\

\noindent\textbf{Primes that split in $KF$.}  Let $\p\in \Sigma_K-S_E$ be a prime ideal that splits in $KF$; i.e., there are two distinct primes $\Pp_1,\Pp_2 \in \Sigma_{KF}$ lying over $\p$.  The maps $E(\FF_\p)\to E(\FF_{\Pp_i})$ are group isomorphisms.  So we have
\begin{align*}
& |\{ \p \in \Sigma_K(x)-S_E : \p \text{ splits in } KF, \, |E(\FF_{\p})|/t \text{ is prime}\}| \\
=& \frac{1}{2}|\{ \Pp \in \Sigma_{KF}(x)-S_{E_{KF}} : |E(\FF_{\Pp})|/t \text{ is prime}\}| + O(\sqrt{x})
=\frac{1}{2}P_{E_{KF},t}(x) + O(\sqrt{x})
\end{align*}
Therefore Conjecture~\ref{C:main} implies that
\begin{align} \label{E:CM split}
 |\{ \p \in \Sigma_K(x)-S_E : \p \text{ splits in } KF, \, |E(\FF_{\p})|/t \text{ is prime}\}| \sim \frac{\C_{E_{KF},\, t}}{2} \frac{x}{(\log x)^2}
\end{align}
as $x\to \infty$, and the constant $\C_{E_{KF},\,t}$ can be computed as in Proposition~\ref{P:CM how to} (if $\C_{E_{KF},\,t}=0$, then there is a congruence obstruction and the left hand side of (\ref{E:CM split}) is indeed bounded).\\

\noindent\textbf{Primes that are inert in $KF$.}  Let $\p\in\Sigma_K-S_E$ be a prime that is inert in $KF$; i.e., $\p\OO_{KF}$ is a prime ideal of $\OO_{KF}$.  For these primes we always have $|E(\FF_{\p})|=N(\p)+1$, so 
\begin{align*}
& |\{ \p \in \Sigma_K(x)-S_E : \p \text{ is inert in } KF, \, |E(\FF_{\p})|/t \text{ is prime}\}| \\
=& |\{ \p \in \Sigma_K(x): \p \text{ is inert in } KF, \, (N(\p)+1)/t \text{ is prime}\}| + O(1);
\end{align*}
Our conjecture combined with the split case above imply that
\begin{equation} \label{E:CM inert}
|\{ \p \in \Sigma_K(x): \p \text{ is inert in } KF, \, (N(\p)+1)/t \text{ is prime}\}| \sim C \frac{x}{(\log x)^2}
\end{equation}
as $x\to\infty$ where $C=\C_{E,t}- \C_{E_{KF},\,t}/2$.  We can also give the more intrinsic definition
\[
C = \lim_{Q\to +\infty} \frac{\delta'_{t}(t\prod_{\ell\leq Q} \ell)}{\prod_{\ell\leq Q}(1-1/\ell)}
\]
where $\delta'_t(m)$ is the density of the set of $\p\in \Sigma_K$ for which $\p$ is inert in $KF$ and $(N(\p)+1)/t$ is invertible modulo $m/\gcd(t,m)$.   The asymptotics of (\ref{E:CM inert}) depends only on $K$ and $KF$, and not the specific curve $E$; we will not consider this case any further.

\subsection{Heuristics}   \label{SS:heuristics}
We will now give a crude heuristic for Conjecture~\ref{C:main} (one could also give a more systematic heuristic as in \cite{Lang-Trotter}).

The prime number theorem states the number of rational primes less than $x$ is asymptotic to $x/\log x$ as $x\to\infty$.  Intuitively, this means that a random natural number $n$ is prime with probability ${1}/{\log n}$.  This probabilistic model, called \emph{Cram\'er's model}, is useful for making conjectures.  Of course the event ``$n$ is prime'' is deterministic (i.e., has probability 0 or 1).

\emph{If} the primality of the integers in the sequence $\{ |E(\FF_\p)|/t\}_{\p\in\Sigma_K-S_E}$ were assumed to behave like random integers, then the likelihood that $|E(\FF_\p)|/t$ is prime would be
\[
 \frac{1}{\log \big(|E(\FF_\p)|/t\big)} \approx \frac{1}{\log( N(\p)+1) - \log t }
 \]
(the last line is reasonable because of Hasse's bound, $\Big||E(\FF_\p)|-(N(\p)+1)\Big|\leq 2\sqrt{N(\p)}$).

However, the $|E(\FF_\p)|/t$ are certainly not random integers with respect to congruences (in particular, they might not all be integers!).   To salvage our model, we need to take into account these congruences.   Fix a positive integer $m$ which we will assume is divisible by $t\prod_{\ell|t}\ell$.  For all but finitely many $\p$, if $|E(\FF_\p)|/t$ is prime then it is invertible modulo $m$.  The density of $\p\in \Sigma_K-S_E$ for which $|E(\FF_\p)|/t$ is an integer and invertible modulo $m$ is $\delta_{E,t}(m)$, while the density of the set of natural numbers that are invertible modulo $m$ is $\prod_{\ell|m}(1-1/\ell)$.  By taking into account the congruences modulo $m$, we expect 
\[
\frac{\delta_{E,t}(m)}{\prod_{\ell|m}(1-1/\ell)}\cdot \frac{1}{\log(N(\p)+1)-\log t}
\]
to be a better approximation for the probability that $|E(\FF_\p)|/t$ is prime for a ``random'' $\p\in\Sigma_K-S_E$.  Taking into account all possible congruences, our heuristics suggest that $|E(\FF_\p)|/t$ is prime for a ``random'' $\p\in\Sigma_K-S_E$ with probability
\[
\C_{E,t}\cdot \frac{1}{\log(N(\p)+1)-\log t}
\]
where 
\[
 \C_{E,t} = \lim_{Q\to +\infty} \frac{\delta_{E,t}(t\prod_{\ell\leq Q}\ell)}{\prod_{\ell\leq Q}(1-1/\ell)}.
\]
We have already seen that this limit converges.

Using our heuristic model, the expected number of $\p\in\Sigma_{K}(x)-S_E$ such that $|E(\FF_\p)|/t$ is prime, should then be well approximated by
\[
\sum_{\substack{\p\in \Sigma_{K}(x)-S_E \\  N(\p)\geq t}} \frac{\C_{E,t}}{ \log (N(\p)+1) -\log t} \sim \C_{E,t}\int^x_{t+1} \frac{1}{ \log (u+1) -\log t} \frac{du}{\log u}.
\]
The restriction of $\p$ in the above sum to those with $ N(\p)\geq t$ is included simply to ensure that each term of the sum is well-defined and positive.  The integral expression follows from the prime number theorem for the field $K$, and is asymptotic to $x/(\log x)^2$.  We can now conjecture that
\[
P_{E,t}(x)\sim \C_{E,t}\int^x_{t+1} \frac{1}{ \log (u+1) -\log t} \frac{du}{\log u}
\]
as $x\to \infty$.

\begin{remark}  In the setting of Conjecture~\ref{C:original} with $t=1$, Koblitz assumed that the divisibility conditions were independent and hence his constant was $\displaystyle \prod_{\ell} \frac{\delta_{E,1}(\ell)}{1-1/\ell}$
 \end{remark}

%*********************************
\section{Common factor of the $|E(\FF_\p)|$} 
\label{S:common factor}

Let $E$ be an elliptic curve over a number field $K$.  
There may be an integer greater than one which divides almost all of the $|E(\FF_\p)|$; this is an obvious obstruction to the primality of the values $|E(\FF_\p)|$.  Thus it will be necessary to divide by this common factor before addressing any questions of primality.  In this section we describe the common factor and explain how it arises from the global arithmetic of $E$.

The following well-known result says that the $K$-rational torsion of $E$ injects into $E(\FF_\p)$ for almost all $\p$ (for a proof see \cite{KatzInv}*{Appendix}).  Define the finite set
\[
\widetilde{S}_{E}:=S_E \cup \{ \p \in \Sigma_K :   e_\p \geq p-1 \text{ where $\p$  lies over the prime $p$}\}
\]
where $e_\p$ is the ramification index of $\p$ over $p$.
\begin{lemma} \label{torsion injects}
For all $\p\in \Sigma_K-\widetilde{S}_{E}$, reduction modulo $\p$ induces an injective group homomorphism
\[
E(K)_{\tors}\hookrightarrow E(\FF_\p).
\]
In particular, $|E(K)_{\tors}|$ divides $|E(\FF_\p) |$ for all $\p\in\Sigma_K-\widetilde{S}_{E}$.
\end{lemma}

The  integer $|E(\FF_\p)|$ is a $K$-isogeny invariant of the elliptic curve $E$.   So for all $\p \in \Sigma_K - \widetilde{S}_{E}$, we find that $|E(\FF_\p)|$ is divisible by
\begin{equation} \label{E:official t}
t_{E} := \lcm_{E'} |E'(K)_{\tors}|,
\end{equation}
where $E'$ varies over all elliptic curves that are isogenous to $E$ over $K$.  One can also show that
\begin{equation}  \label{E:sup} 
t_{E} = {\max}_{E'} |E'(K)_{\tors}|. 
\end{equation}
From our discussion above, $t_{E}$ divides $|E(\FF_\p)|$ for almost all $\p\in\Sigma_K$ (in particular, Conjecture~\ref{C:main} is only interesting when $t_E$ divides $t$).  The following theorem of Katz shows that $t_{E}$ is the largest integer with this property.

\begin{thm}[Katz \cite{KatzInv}*{Theorem 2(bis)}] 
\label{T:Katz theorem} 
Let $\Sigma$ be a subset of $\Sigma_K -\widetilde{S}_{E}$ with density $1$. 
Then
\[
t_{E} = \gcd_{\p\in \Sigma} |E(\FF_\p)|.
\]
\end{thm}   

There is an elliptic curve $E'$ which is $K$-isogenous to $E$ satisfying $t_{E}=|E'(K)_{\tors}|$.  Thus our conjecture with $t=t_E$ predicts how frequently the groups $E'(\FF_\p)/E'(K)_{\tors}$ have prime cardinality as $\p$ varies (this was mentioned by Koblitz in the final remarks of \cite{Koblitz} as a natural way to generalize his paper).  Koblitz's original conjecture was restricted to those elliptic curves over $\QQ$ with $t_E=1$.

\begin{remark}
Using the characterization of $t_{E}$ from Theorem~\ref{T:Katz theorem}, we can also express $t_{E}$ in terms of our Galois representations.  It is the largest integer $t$ such that $\det(I-\rho_{t}(g)) \equiv 0 \bmod{t}$ for all $g\in G(t)$.
\end{remark}

%************************************************************
\section{Serre Curves} 
\label{S:Serre curve section}

\subsection{The constant $\C_{E,1}$ for Serre curves}
Throughout this section, we assume that $E$ is a elliptic curve over $\QQ$ without complex multiplication.  For each $m\geq1$, we have defined a Galois representation $\rho_m\colon \Gal(\Qbar/\QQ) \to \Aut(E[m]).$  Combining them all together, we obtain a single representation
\[
 \widehat{\rho} \colon \Gal(\Qbar/\QQ) \to \Aut(E_{\tors}) \cong \GL(\Zhat).
\]
A theorem of Serre \cite{SerreInv} says that that the index of $G(m)$ in $\Aut(E[m])$ is bounded by a constant that depends only on $E$; equivalently, $\widehat{\rho}(\Gal(\Qbar/\QQ))$ has finite index in $\GL(\Zhat)$.

Serre has also shown that the map $\widehat{\rho}$ is \emph{never} surjective \cite{SerreInv}*{Proposition 22}.  He proves this by showing that $\widehat{\rho}(\Gal(\Qbar/\QQ))$ lies in a specific index $2$ subgroup $H_E$ of $\Aut(E_{\tors})$ (see \S\ref{SS:the group H} for details).  Following Lang and Trotter, we make the following definition.
\begin{definition}  An elliptic curve $E$ over $\QQ$ is a \emph{Serre curve} if $\widehat{\rho}(\Gal(\Qbar/\QQ))$ is an index $2$ subgroup of $\Aut(E_{\tors})$.
\end{definition}

Serre curves are thus elliptic curves over $\QQ$ whose Galois action on their torsion points are as ``large as possible''.  For examples of Serre curves, see \S\ref{S:main example} and \cite{SerreInv}*{\S5.5}.   Jones has shown that ``most'' elliptic curves over $\QQ$ are Serre curves \cite{Jones-AAECASCF}.  Thus Serre curves are prevalent and we have a complete understanding of the groups $G(m)$ (see below); thus they are worthy of special consideration.  We are particularly interested in Conjecture~\ref{C:main} with $t=1$.

\begin{prop} \label{P:Serre curves constant}
Let $E/\QQ$ be a Serre curve.   Let $D$ be the discriminant of the number field $\QQ(\sqrt{\Delta})$ where $\Delta$  is the discriminant of any Weierstrass model of $E$ over $\QQ$.  Then
\[
\C_{E,1} = 
\begin{cases} 
\displaystyle \mathfrak{C}\bigg(1 +  \prod_{\ell | D }  \frac{ 1}{   \ell^3 - 2\ell^2 - \ell + 3  }\bigg)    & \text{ if $D\equiv 1 \mod{4}$,}\\
\mathfrak{C} & \text{ if $D \equiv 0 \mod{4}$}
\end{cases}
\]
where $\mathfrak{C}=\prod_\ell  \Bigl(1- \frac{\ell^2-\ell-1}{(\ell-1)^3(\ell+1)} \Bigr).$
\end{prop}

The proof of Proposition~\ref{P:Serre curves constant} will be given in \S\ref{SS:Serre curve calculations}.

\begin{remark}
\begin{romanenum}
\item
 In the paper \cite{Jones-Ann}, Jones studies the constant $\C_{E,1}$ as $E/\QQ$ varies over certain families of elliptic curves.  The ``main term'' of his results comes from the contribution of the Serre curves.
\item  
There are elliptic curves defined over number fields $K\neq \QQ$ such that $\widehat\rho(\Gal(\Kbar/K))=\Aut(E_{\tors})$ is surjective!  The first example was give by A.~Greicius \cite{Aaron} (also see \cite{Zywina-maximalgalois}).
\end{romanenum}
\end{remark}

%****************************************************
\subsection{The group $H_E$} \label{SS:the group H}

We shall now describe the desired group $H_E$ (see \cite{SerreInv}*{p.~311} for further details).  
Let $D$ be the discriminant of the number field $L:=\QQ(\sqrt{\Delta})$ where $\Delta$  is the discriminant of any Weierstrass model of $E$ over $\QQ$ (note that $L$ is independent of the choice of model).  Define the character
\[
\chi_D \colon \Gal(\Qbar/\QQ) \twoheadrightarrow \Gal(L/\QQ) \hookrightarrow\{ \pm 1 \},
\]
where the first map is restriction.  

The field $L$ is contained in $\QQ(E[2])$.  Let $\varepsilon \colon \Aut(E[2]) \to \{\pm 1\}$ be the character which corresponds to the signature map under any isomorphism $\Aut(E[2])\cong \mathfrak{S}_3$. One checks that
$\chi_D(\sigma) = \varepsilon(\rho_2(\sigma))$
for all $\sigma \in \Gal(\Qbar/\QQ)$.

Since $L$ is an abelian extension of $\QQ$, it must lie in a cyclotomic extension\footnote{This is where the assumption $K=\QQ$ is important} of $\QQ$.   Set $d:=|D|;$ it is the smallest positive integer for which $L \subseteq \QQ(\zeta_d)$ where $\zeta_d\in\Qbar$ is a primitive $d$-th root of unity.  The homomorphism $\det \circ \rho_{d} : \Gal(\Qbar/\QQ) \to (\ZZ/d\ZZ)^\times$ factors through the usual isomorphism $\Gal(\QQ(\zeta_{d})/\QQ) \xrightarrow{\sim} (\ZZ/d\ZZ)^\times$.  Thus there exists a unique character $\alpha\colon (\ZZ/d\ZZ)^\times \to \{\pm 1\}$ such that 
$\chi_D(\sigma) =  \alpha( \det \rho_{d}(\sigma))$
for all $\sigma \in \Gal(\Qbar/\QQ)$.  The minimality of $d$ implies that $\alpha$ is a primitive Dirichlet character of conductor $d$.

Combining our two descriptions of $\chi_D$, we have $\varepsilon( \rho_2(\sigma) )=\alpha(\det \rho_{d}(\sigma))$ for all $\sigma \in \Gal(\Qbar/\QQ)$.  
Define the integer $M_E:=\lcm(d,2)$, and the group
\[
 H(M_E):=\big\{ g\in \Aut(E[M_E]) : \varepsilon(A \bmod{2}) = \alpha(\det(A\bmod{d})) \big\}.
\]
which has index $2$ in $\Aut(E[M_E])$.  By the above discussion, $H(M_E)$ contains $G(M_E)$.  The index $2$ subgroup $H_E$ of $\Aut(E_{\tors})$ mentioned earlier is just the inverse image of $H(M_E)$ under the natural map $\Aut(E_{\tors})\to \Aut(E[M_E])$, and $E$ is a Serre curve if and only if $\widehat{\rho}(\Gal(\Qbar/\QQ))=H_E$.

\begin{prop} \label{P:Serre curve explicit}
Let $E/\QQ$ be a Serre curve, and let $m$ be a positive integer.  If $M_E| m$, then the group $G(m)$ is the inverse image of $H(M_E)$ under the natural map $\Aut(E[m]) \to \Aut(E[M_E])$.  If $M_E\nmid m$, then $G(m)=\Aut(E[m])$.
\end{prop}
\begin{proof}
This is a purely group theoretic statement which we leave to the reader.   If $\alpha\colon(\ZZ/d\ZZ)^\times \to \{\pm 1\}$ is a primitive Dirichlet character, define the group
\[
H := \{ A \in \GL(\Zhat) : \varepsilon(A \bmod{2}) = \alpha(\det(A \bmod d))\}.
\]
The proposition simply says that $H \bmod{m} = \GL(\ZZ/m\ZZ)$ if and only if $d|m$ and $2|m$.
\end{proof}

\subsection{Proof of Proposition~\ref{P:Serre curves constant}} 
\label{SS:Serre curve calculations}

Let $E/\QQ$ be a Serre curve, and keep the notation introduced in \S\ref{SS:the group H}. \\

Let's first consider the case where $D\equiv 0 \bmod{4}$.  The integer $M_E=d=|D|$ is divisible by $4$, so by Proposition~\ref{P:Serre curve explicit}, we have $G(m)=\Aut(E[m])$ for all \emph{squarefree} $m$.  By Proposition~\ref{P:how to}, with $t=1$ and $M=1$, we have $\C_{E,1} = \prod_\ell  \big(1- \frac{\ell^2-\ell-1}{(\ell-1)^3(\ell+1)} \big)$.\\

We shall now restrict to the case where $D \equiv 1 \bmod{4}$.  In this case, the integer $M_E=\lcm(2,d)=2d=2|D|$ is squarefree. By Proposition~\ref{P:Serre curve explicit}, we have $G(M_E \cdot m)=H(M_E)\times \Aut(E[m])$ for all \emph{squarefree} $m$ relatively prime to $M_E$.  Thus by Proposition~\ref{P:how to}, with $t=1$ and $M=M_E$, we have 
\begin{equation} \label{E:Serre curve independence}
 \C_{E,1} = \frac{|H(M_E)\cap \Psi_1(M_E)|/|H(M_E)|}{\prod_{\ell|M_E}(1-1/\ell) } \prod_{\ell \nmid M_E}  \Big(1- \frac{\ell^2-\ell-1}{(\ell-1)^3(\ell+1)} \Big).
\end{equation}
Since $d$ is odd and $\alpha$ is a quadratic character of conductor $d$, $\alpha$ is the Jacobi symbol $\left(\tfrac{\cdot}{d} \right)$.  The set $H(M_E)\cap \Psi_1(M_E)$ then has the same cardinality as the set
\[
X=\big\{ A \in \GL(\ZZ/M_E\ZZ) :  \varepsilon(A \bmod{2})= \left( \tfrac{\det(A\bmod d)}{d}\right), \; \det( I - A ) \in (\ZZ/M_E\ZZ)^\times \big\}.
\]
Take any element $A\in X$.  Setting $A_2:=A \bmod{2}$, we have $\det(I-A_2) = 1$ and $\det(A_2)=1$ in $\ZZ/2\ZZ$.  The only matrices in $\GL(\ZZ/2\ZZ)$ that satisfy these conditions are: 
$\big(\begin{smallmatrix}1 & 1 \\1 & 0\end{smallmatrix}\big)$ and $\big(\begin{smallmatrix}0 & 1 \\1 & 1\end{smallmatrix}\big).$
These two matrices have order 3 in $\GL(\ZZ/2\ZZ)$, and hence $\varepsilon(A_2)= 1$.  Since $d$ is odd, $H(M_E)\cap \Psi_1(M_E)$ has twice as many element as the set
\[
Y = \{ A \in \GL(\ZZ/d\ZZ) :  \left( \tfrac{\det(A)}{d}\right) =1, \det( I - A ) \in (\ZZ/d\ZZ)^\times \}.
\]
For each prime $\ell|d$, define the sets
\[
Y_\ell^{\pm} := \big\{ A \in \GL(\ZZ/\ell\ZZ) :  \big(\tfrac{\det(A)}{\ell}\big)= \pm 1,\; \det( I - A )\neq 0 \big\}.
\]
Under the isomorphism  $\GL(\ZZ/ d\ZZ) \cong \prod_{\ell | d} \GL(\ZZ/\ell\ZZ)$, the set $Y$ corresponds to the disjoint union of sets:
\[
 \bigcup_{ \substack{ J \subseteq \{ \ell: \ell | d \} \\ |J| \text{ even}}} \;\prod_{\ell \in J} Y_\ell^- \times \prod_{\ell|d,\, \ell\not\in J} Y_\ell^+.
 \]
Therefore,
\begin{align} \label{E:Serre curve crunch}
|H(M_E)\cap \Psi_1(M_E)| = 2|Y| &= 2\sum_{ f | d } \frac{1+\mu(f)}{2} \prod_{\ell | f} |Y_\ell^-| \prod_{\ell | \frac{d}{f} } |Y_\ell^+|\\
\notag &=  \prod_{\ell | d} (|Y_\ell^+| + |Y_\ell^-|)  +  \prod_{\ell | d} (|Y_\ell^+| - |Y_\ell^-|),
\end{align}
where $\mu$ is the M\"obius function.

\begin{lemma} \label{L:odd Y factors}
For $\ell | d$, 
\[
\frac{|Y_\ell^+| + |Y_\ell^-|}{|\GL(\ZZ/\ell\ZZ)| (1-1/\ell)} =  1 - \frac{\ell^2-\ell-1}{(\ell-1)^3(\ell+1)}\quad \text{ and } \quad \frac{|Y_\ell^+| - |Y_\ell^-|}{|\GL(\ZZ/\ell\ZZ)| (1-1/\ell)}=  \frac{\ell}{(\ell-1)^3(\ell+1)}.
\]
\end{lemma}
\begin{proof}
We use Lemma~\ref{L:eigenvalues} to compute the $|Y_\ell^\pm|$:
\begin{align*} 
|Y_\ell^\pm| & 
= |\{ A \in \GL(\ZZ/\ell\ZZ) :  \left(\tfrac{\det A}{\ell}\right) = \pm 1 \}| 
- |\{ A \in \GL(\ZZ/\ell\ZZ) :  \left(\tfrac{\det A}{\ell}\right)= \pm 1,\; \det( I - A ) =0 \}| \\
&= \frac{1}{2}|\GL(\ZZ/\ell\ZZ)|- \sum_{a\in \FF_\ell^\times,\, \left(\tfrac{a}{\ell}\right)=\pm 1} |\{ A \in \GL(\FF_\ell) :  \text{the eigenvalues of $A$ are $1$ and $a$} \}| \\
&= \frac{1}{2}\ell(\ell-1)^2(\ell+1)- \frac{\ell-1}{2}(\ell^2+\ell) + \frac{1}{2}(1\pm 1) \ell
\end{align*}
The rest is a direct calculation.
\end{proof}
Using (\ref{E:Serre curve crunch}), Lemma~\ref{L:odd Y factors}, and $|H(M_E)|=\frac{1}{2}\prod_{\ell|M_E} |\GL(\ZZ/\ell\ZZ)|= 3 \prod_{\ell|d} |\GL(\ZZ/\ell\ZZ)|$, we have:
\begin{align*} 
\frac{|H(M_E)\cap \Psi_1(M_E)|/|H(M_E)|}{\prod_{\ell|M_E}(1-1/\ell) } 
&= \frac{1}{3\cdot\frac{1}{2}}\Big(  \prod_{\ell | d} \frac{|Y_\ell^+| + |Y_\ell^-|}{|\GL(\ZZ/\ell\ZZ)| (1-1/\ell)}  +  \prod_{\ell | d} \frac{|Y_\ell^+| - |Y_\ell^-|}{|\GL(\ZZ/\ell\ZZ)| (1-1/\ell)} \Big)\\
&= \frac{2}{3}\bigg(  \prod_{\ell | d} \Big(1 - \frac{\ell^2-\ell-1}{(\ell-1)^3(\ell+1)}\Big)  +  \prod_{\ell | d} \frac{\ell}{(\ell-1)^3(\ell+1)} \bigg)\\
&= \frac{2}{3} \prod_{\ell | d} \Big(1 - \frac{\ell^2-\ell-1}{(\ell-1)^3(\ell+1)}\Big) \Bigg(  1 + 
   \prod_{\ell | d }  \dfrac{  \frac{\ell }{(\ell-1)^3(\ell+1)}}{   1- \frac{\ell^2-\ell-1}{(\ell-1)^3(\ell+1)}    } \Bigg) \\
 &= \prod_{\ell | M_E=2d} \Big(1 - \frac{\ell^2-\ell-1}{(\ell-1)^3(\ell+1)}\Big) \Bigg(  1 + 
   \prod_{\ell | d } \frac{1}{\ell^3 - 2\ell^2 - \ell + 3} \Bigg).
\end{align*}
Proposition~\ref{P:Serre curves constant} follows by combining this expression with (\ref{E:Serre curve independence}) and noting that $d=|D|$.

%********************************************************
\section{Example: $y^2 =x^3+6x-2$} \label{S:main example}
In this section, we consider the elliptic curve $E$ over $\QQ$ defined by the Weierstrass equation $y^2=x^3+6x-2$.     
%   E= ellinit([0,0,0,6,-2]);
This curve is a Serre curve; for a proof, see \cite{Lang-Trotter}*{Part I \S7}. 

The given Weierstrass model has discriminant $\Delta=-2^6 3^5$, and hence $\QQ(\sqrt{\Delta})$ has discriminant $-3$.  By Proposition \ref{P:Serre curves constant} and (\ref{E:numerical constant}),
\begin{equation} \label{E:main example C}
\C_{E,1} = \frac{10}{9} \prod_{\ell}  \Bigl(1- \frac{\ell^2-\ell-1}{(\ell-1)^3(\ell+1)} \Bigr)
\approx 0.5612957424882619712979385\ldots
\end{equation}
In the following table, the ``expected number'' of $p\leq x$ with $p\nmid 6$ such that $|E(\FF_p)|$ is prime is
\begin{equation} \label{E:prediction}
\C_{E,1}\int_{2}^x \frac{1}{\log(u+1) }\frac{du}{\log u}
\end{equation}
rounded to the nearest integer.

\begin{table}[H]
\caption{Number of $p\leq x$ with $p\nmid 6$ such that $|E(\FF_p)|$ is prime.}
\begin{tabular}{lcc | lcc} \label{T:1}
        %\hline
	$x$
        & Actual
        & Expected
        & $x$
        & Actual
        & Expected
        \\ \hline \hline
20000000    &45285    &45592 & 520000000 & 810038 &810610\\
40000000    &83272    &83564 & 540000000 & 837904 & 838429\\
60000000    &118991   &119317& 560000000 & 865500 &866145\\
80000000    &153257   &153735& 580000000 & 893592 &893763\\
100000000   &186727   &187209& 600000000 & 921156 &921287\\
120000000   &219604   &219958& 620000000 & 948710 &948720\\
140000000   &251728   &252123& 640000000 & 975828 &976066\\
160000000   &283381   &283799& 660000000 & 1003310 &1003328\\
180000000   &314686   &315058& 680000000 & 1030626 &1030508\\
200000000   &345255   &345953& 700000000 & 1057836 &1057610\\
220000000   &375910   &376526& 720000000 & 1084734 &1084636\\
240000000   &406162   &406810& 740000000 & 1111877 &1111589\\
260000000&  436059&   436833 & 760000000 & 1138685 &1138470\\
280000000&  465712&   466619 & 780000000 & 1165267 &1165282\\
300000000&  495338&   496186 & 800000000 & 1192027 &1192027\\
320000000&  524820&   525552 & 820000000 & 1218668 &1218707\\
340000000&  553850&   554731 & 840000000 & 1245563 &1245324\\
360000000&  583047&   583736 & 860000000 & 1272004 &1271878\\
380000000&  611978&   612577 & 880000000 & 1298490 &1298373\\
400000000&  640571&   641265 & 900000000 & 1324972 &1324810\\
420000000&  668855&   669809 & 920000000 & 1351413 &1351190\\
440000000&  697006&   698216 & 940000000 & 1377897 &1377514\\
460000000&  725494&   726493 & 960000000 & 1404065 &1403784\\
480000000&  753548&   754648 & 980000000 & 1430213 & 1430001\\
500000000&  781819&   782685 & 1000000000 & 1456288 &1456166\\
 \hline
\end{tabular}
\end{table}  

\begin{remark} 
The predicted constant in \cite{Koblitz} was $9/10\cdot \C_{E,1}$.  This  would have led to a predicted value of $\approx 1310549$ in the last entry of Table~\ref{T:1}.
\end{remark}

%****************************************************************************
\section{Example: $y^2 = x^3 + 9x+18$}
\label{S:Jones example}
Let $E$ be the elliptic curve over $\QQ$ defined by the Weierstrass equation $y^2 = x^3 + 9x+18$.  The discriminant of our Weierstrass model is $\Delta=-2^8 3^6$.  This is the curve mentioned in \S\ref{SS:an example}.  It is not isogenous over $\QQ$ to a curve with nontrivial $\QQ$-torsion (in the notation of \S\ref{S:common factor}, $t_E=1$), but we have $\C_{E,1}=0$.   We saw that $|E(\FF_p)|$ was divisible by $3$ if $p\equiv 1\bmod{4}$ and divisble by $2$ if $p\equiv 3\bmod{4}$.   In this section we will give numerical evidence for Conjecture~\ref{C:main} with $t \in \{2,3,6\}$.\\

We now state, without proof, enough information about the groups $G(m)$ so that one may compute the constants $\C_{E,2}$, $\C_{E,3}$ and $\C_{E,6}$.  \\

\noindent $\bullet$ \textbf{$2$-torsion.} We have $G(2)=\Aut(E[2])$.  

\noindent $\bullet$ \textbf{$4$-torsion.} Viewing $\Aut(E[2])$ as the symmetric group on $E[2]-\{0\}$, let $\varepsilon\colon\Aut(E[4])\to  \Aut(E[2])\to \{\pm 1\}$ be the signature homorphism.  Let $\chi$ be the non-identity character of $(\ZZ/4\ZZ)^\times$.  Then $\QQ(\sqrt{\Delta})=\QQ(i)$ implies that $G(4)$ is contained in the group
\[
\{ A \in \Aut(E[4]): \varepsilon(A) = \chi(\det(A)) \},
\]
and this is actually an equality.   We then have
\[
\rho_4\big(\Gal(\Qbar/\QQ(i))\big)=\{ A \in \Aut(E[4]): \varepsilon(A) = \chi(\det(A))=1 \}.
\]
The maximal abelian extension of $\QQ(i)$ in $\QQ(E[4])$ is $\QQ(i,\alpha,\sqrt{6})$ where $\alpha$ is a root of $x^3+9x+18$. (Group theory with $G(4)$ tells us that it is a degree six extension of $\QQ(i)$.  In general, one always has $\QQ(i,\sqrt[4]{\Delta}) \subseteq \QQ(E[4]);$ for our curve $\QQ(i,\sqrt[4]{\Delta})=\QQ(i,\sqrt[4]{-9})=\QQ(i,\sqrt{6})$.)

\noindent $\bullet$ \textbf{$3$-torsion.}  Choose a $\ZZ/3\ZZ$-basis of $E[3]$ whose first vector is $P:=(-3,6i)$.  Then with respect to this basis, $G(3)=\rho_3(\Gal(\Qbar/\QQ))$ is the subgroup of upper triangular matrices in $\GL(\ZZ/3\ZZ)$.  

\noindent $\bullet$ \textbf{$9$-torsion.}  The group $G(9)$ is the inverse image of $G(3)$ under the map $\Aut(E[9])\to \Aut(E[3])$.  The maximal abelian extension of $\QQ(i)$ in $\QQ(E[9])$ is $\QQ(i,\zeta_9)$.  Let $\beta\colon G(9) \to \{\pm1\}$ be the homomorphism for which $\sigma(P)=\beta(\rho_9(\sigma))P$ for all $\sigma \in \Gal(\Qbar/\QQ)$.

\noindent $\bullet$ \textbf{$36$-torsion.}  We may view $G(36)$ as a subgroup of $G(4)\times G(9)$, where we have already described $G(4)$ and $G(9)$.  To work out $G(36)$, one needs to know the field $\QQ(E[4])\cap \QQ(E[9])$.  We claim that $\QQ(E[4])\cap \QQ(E[9])=\QQ(i)$.
Suppose that $\QQ(E[4])\cap \QQ(E[9])\supsetneq\QQ(i)$; then the solvability of $G(4)$ implies that there is a nontrivial abelian extension $L/\QQ(i)$ in $\QQ(E[4])\cap \QQ(E[9])$.  However the maximal abelian extension of $\QQ(i)$ in $\QQ(E[4])$ and $\QQ(E[9])$ is $\QQ(i,\alpha,\sqrt{6})$ and $\QQ(i,\zeta_9)$, respectively. Thus $L\subseteq \QQ(i,\alpha,\sqrt{6}) \cap \QQ(i,\zeta_9)=\QQ(i).$  We deduce that 
\[
G(36) = \big\{ (A,B) \in \Aut(E[4])\times \Aut(E[9]) : \varepsilon(A) = \chi(\det(A))= \beta(B) \big\}
\]
and 
\[
\rho_{36}(\Gal(\Qbar/\QQ)) = \big\{ (A,B) \in \Aut(E[4])\times \Aut(E[9]) : \varepsilon(A) = \chi(\det(A))= \beta(B)=1 \big\}.
\]

\noindent $\bullet$ \textbf{$5$-torsion.}  The group $G(5)$ is the unique subgroup of $\Aut(E[5])$ of order $96$.  The image in $\operatorname{PGL}_2(\ZZ/5\ZZ)$ is isomorphic to the symmetric group $S_4$ (this is one of the exceptional cases in \cite{SerreInv}*{Prop.~16}).  The maximal abelian extension of $\QQ$ in $\QQ(E[5])$ is $\QQ(\zeta_5)$.

\noindent $\bullet$ \textbf{$\ell$-torsion, $\ell\geq 7$.}
For every prime $\ell\geq 7$, we have $G(\ell)=\Aut(E[\ell])$.  Group theory shows that for any squarefree positive integer $m$ relatively prime to $2\cdot3\cdot 5$, we have $G(m)={\prod}_{\ell|m} \Aut(E[\ell])$.  The maximal abelian extension of $\QQ$ in $\QQ(E[m])$ is $\QQ(\zeta_m)$.

\noindent $\bullet$ For any squarefree positive integer $m$ relatively prime to $2\cdot3\cdot 5$, we claim that
\begin{equation} \label{E:Jones example galois}
G(36\cdot 5\cdot {\prod}_{\ell|m} \ell) = G(36)\times G(5) \times {\prod}_{\ell|m} \Aut(E[\ell]).
\end{equation}
Since $G(36)$ and $G(5)$ are solvable, it suffices to show that the maximal abelian extensions of $\QQ$ in $\QQ(E[36])$, $\QQ(E[5])$, and $\QQ(E[m])$ are pairwise linearly disjoint over $\QQ$ (this is clear since the intersection of any two of these fields is an unramified extension of $\QQ$).\\

Take any $t\in \{2,3,6\}$.  From the above description, we may apply Proposition~\ref{P:how to} with $M=30$ to obtain
\[
 \C_{E,t} = \dfrac{ \delta_{E,t}(36\cdot 5)}{\prod_{\ell|2\cdot3\cdot 5} (1-1/\ell)} \prod_{\ell \geq 7} \Big(1 - \frac{\ell^2-\ell-1}{(\ell-1)^3(\ell+1)}\Big).
\]
Since $G(36\cdot5)=G(36)\times G(5)$, we have
\[
 \C_{E,t} = \frac{\delta_{E,t}(36)}{(1-1/2)(1-1/3)} \frac{\delta_{E,t}(5)}{1-1/5} \prod_{\ell\geq 7} \Bigl(1 - \frac{\ell^2-\ell - 1}{(\ell-1)^3(\ell+1)}  \Bigr).
\]
We have $\delta_{E,t}(5)=\delta_{E,1}(5)$ since $t$ is relatively prime to $5$, and using our description of $G(5)$ one can show that 
$\delta_{E,t}(5)=\delta_{E,1}(5)=77/96$.  Hence
\[
 \C_{E,t} = \delta_{E,t}(36) \frac{1232}{219} \prod_{\ell} \Bigl(1 - \frac{\ell^2-\ell - 1}{(\ell-1)^3(\ell+1)}  \Bigr).
\]
Using our description of $G(36)$, one can show that
$\delta_{E,2}(36)=1/8$, $\delta_{E,3}(36)=5/27$, and $\delta_{E,6}(36)= 1/12$.  We record the resulting constants in the next lemma.

\begin{lemma} \label{L:constant claimed} For the elliptic curve $E$ over $\QQ$ defined by $y^2=x^3+9x+18$, we have
\[
\C_{E,2} = \frac{154}{219} \mathfrak{C} \quad \C_{E,3} = \frac{6160}{5913} \mathfrak{C}  \quad \C_{E,6} = \frac{308}{657} \mathfrak{C} 
\]
where $\mathfrak{C}=\prod_{\ell} \Bigl(1 - \frac{\ell^2-\ell - 1}{(\ell-1)^3(\ell+1)}  \Bigr)$. 
\end{lemma}

In the following table, the ``expected number'' of $p\leq x$ with $p\nmid 6$ such that $|E(\FF_p)|/t$ is prime is
\begin{equation} \label{E:prediction}
\C_{E,t}\int_{t+1}^x \frac{1}{\log(u+1) }\frac{du}{\log u}
\end{equation}
rounded to the nearest integer, where $\C_{E,t}$ is estimated using Lemma~\ref{L:constant claimed} and (\ref{E:numerical constant}).

\begin{table}[H]
\caption{Number of $p\leq x$ with $p\nmid 6$ such that $|E(\FF_p)|/t$ is prime.}
\begin{tabular}{l|cc|cc|cc} %\label{T:1}
        %\hline
      
     &  $t=2$ & & $t=3$& & $t=6$\\     
 $x$ & Actual & Expected & Actual & Expected & Actual & Expected\\ 
\hline \hline
40000000    &55118  &55244  &83736   &84036   &39554  &39634 \\
80000000    &101556 &101444 &154113  &154134  &72535  &72537 \\
120000000   &145334 &144995 &220046  &220165  &103413 &103490 \\
160000000   &187516 &186949 &283458  &283747  &133307 &133271 \\
200000000   &228440 &227774 &345198  &345597  &161983 &162224 \\
240000000   &268461 &267730 &405675  &406118  &190166 &190543 \\
280000000   &307911 &306986 &464711  &465565  &217926 &218348 \\
320000000   &346499 &345657 &523022  &524117  &245405 &245727 \\
360000000   &384950 &383827 &580584  &581901  &272350 &272739 \\
400000000   &422640 &421560 &637825  &639017  &299112 &299433 \\
440000000   &459555 &458907 &694394  &695541  &325385 &325845 \\
480000000   &496734 &495907 &750663  &751535  &351567 &352005 \\
520000000   &533405 &532594 &806485  &807050  &377507 &377936 \\
560000000   &570295 &568996 &861533  &862129  &403533 &403659 \\
600000000   &606622 &605135 &916370  &916807  &428958 &429192 \\
640000000   &642830 &641032 &970514  &971114  &454130 &454548 \\
680000000   &678475 &676705 &1024511 &1025079 &479230 &479741 \\
720000000   &713909 &712169 &1077829 &1078722 &504194 &504782 \\
760000000   &749026 &747436 &1130770 &1132066 &529125 &529680 \\
800000000   &784432 &782518 &1183934 &1185128 &553804 &554443 \\
840000000   &819581 &817427 &1236561 &1237925 &578378 &579081 \\
880000000   &854213 &852172 &1288783 &1290470 &603045 &603599 \\
920000000   &888701 &886761 &1341501 &1342777 &627523 &628004 \\
960000000   &923138 &921202 &1393453 &1394859 &651810 &652301 \\
1000000000  &957322 &955502 &1445188 &1446724 &675851 &676497 \\
\hline
\end{tabular}
\end{table}

\section{CM example: $y^2=x^3-x$} \label{S:CM example}

Let $E$ be the elliptic curve over $\QQ$ defined by the Weierstrass equation $y^2=x^3-x$.   This curve has complex multiplication by $R=\ZZ[i]$, where $i$ corresponds to the endomorphism $(x,y)\mapsto (-x, iy)$ defined over $\QQ(i)$.  The curve $E$ has conductor $2^5$.

The torsion group $E(\QQ(i))_{\tors}$ has order $8$ and is generated by $(i,1-i)$ and $(1,0)$.  So for those primes $p$ that split in $\QQ(i)$ (i.e., $p\equiv 1 \bmod{4}$), we find that $|E(\FF_p)|$ is divisible by $8$.  In this section, we give numerical evidence for Conjecture~\ref{C:main} with $t=8$.  We will study it in the form given in (\ref{E:CM split}), which conjectures that
\begin{equation} \label{E:CM working conjecture}
 |\{ p \leq x : p \equiv 1 \bmod{4}, \, |E(\FF_p)|/8 \text{ is prime}\}| \sim \frac{\C_{E_{\QQ(i)},\, 8}}{2} \int^x_9 \frac{1}{\log(u+1)-\log 8}\, \frac{du}{\log u}
 \end{equation}
as $x\to \infty$ (we have used the integral version of the conjecture since it should give a better approximation).  This particular curve was studied by Iwaniec and Jim\'enez Urroz in \cite{Iwaniec} where they proved that
\[
|\{ p \leq x : p \equiv 1 \bmod{4}, \, |E(\FF_p)|/8  \text{ is a prime or a product of two primes}\}| \gg \frac{x}{(\log x)^2}
\]
using sieve theoretic methods.
We now describe the constant $\C_{E_{\QQ(i)},8}$:

\begin{lemma} \label{L:CM example constant}
Let $E$ be the elliptic curve over $\QQ$ given by $y^2=x^3-x$.  Then
\[
\C_{E_{\QQ(i)},8} =  \prod_{\ell\neq 2} \Bigl( 1 - \chi(\ell) \frac{\ell^2-\ell-1}{(\ell-\chi(\ell))(\ell-1)^2} \Bigr) 
\]
where $\chi(\ell)=(-1)^{(\ell-1)/2}$.  We have $\C_{E_{\QQ(i)},8}  \approx 1.067350894$.
\end{lemma}
\begin{proof}
Since $E$ has conductor $2^5$, the curve $E_{\QQ(i)}$ has good reduction away from the prime $(1+i)$.   For the curve $E_{\QQ(i)}$, fix notation as in the proof of Proposition~\ref{P:CMopenimage} (in particular, $K=F=\QQ(i)$ and $R=\ZZ[i]$).  Checking the two conditions in the second half of the proof of Proposition~\ref{P:CMopenimage}, we find that the Proposition holds for $E_{\QQ(i)}$ with $M=2$.  

The discriminant of $\QQ(i)$ is $-4$ and the conductor of the order $R$ is $1$, so by Proposition~\ref{P:CM how to} we have
\[
 \C_{E_{\QQ(i)},8} = \dfrac{ \delta_{E_{\QQ(i)},8}(16) }{ (1-1/2)}  \prod_{\ell \neq 2} \Big(1 - \chi(\ell) \frac{\ell^2-\ell-1}{(\ell-\chi(\ell))(\ell-1)^2}\Big)
\]
where $\chi$ is the Kronecker character of $\QQ(i)$ (and hence $\chi(\ell)=(-1)^{(\ell-1)/2}$).  To prove the required product description of $\C_{E_{\QQ(i)},8}$, it remains to show that $\delta_{E_{\QQ(i)},8}(16)=1/2$.  Consider the representation
\[
\widehat{\rho}_2\colon \Gal(\bbar{\QQ(i)}/\QQ(i)) \to (R\otimes\ZZ_2)^\times = (\ZZ_2[i])^\times
\]
arising from the Galois action on the Tate module $T_2(E)$.  It is well known that $\widehat{\rho}_2$ has image equal to $1+\p_2^3$ where $\p_2$ is the prime ideal $(1+i)\ZZ_2[i]$ (for example, see \cite{KatzSarnak}*{9.4}).  In particular, 
\begin{align*}
G(16)=\rho_{16}\big(\Gal(\bbar{\QQ(i)}/\QQ(i))\big) &= \big(1+\p_2^3\big)/\big(1+16\ZZ_2[i]\big)
 = \big(1+\p_2^3\big)/\big(1+\p_2^{8}\big).
\end{align*}
Under our identification of $\Aut_{\ZZ_2[i]}(T_2(E))$ with $\ZZ_2[i]^\times$, we find that $\det(I-a)$ agrees with $N(1-a)$  where $N$ is the norm map from $ \ZZ_2[i]$ to $\ZZ_2$.   We deduce that $\delta_{E_{\QQ(i)},8}(16)$ is the proportion of $a\in \big(1+\p_2^3\big)/\big(1+\p_2^{8}\big)$ for which $N(1-a) \equiv 8 \bmod{16}$; this is indeed equal to $1/2$.

With respect to how one estimates the constant, we simply note that
\[
\C_{E,8}=L(1,\chi)^{-1}\prod_{\ell\neq 2} \Big( 1 - \chi(\ell) \frac{\ell^2-\ell-1}{(\ell-\chi(\ell))(\ell-1)^2}\Big)\Big(1-\frac{\chi(\ell)}{\ell}\Big)^{-1}.
\] 
The product is now absolutely convergent and $L(1,\chi)=\pi/4$ by the class number formula.
\end{proof}

In the following table, the ``Actual'' column is the value of the left hand side of (\ref{E:CM working conjecture}), while the ``Expected'' column is the right hand side of (\ref{E:CM working conjecture}) with the approximation from Lemma~\ref{L:CM example constant}.

\begin{table}[H]
\caption{Number of $p\leq x$ with $p\equiv 1 \bmod{4}$ such that $|E(\FF_p)|/8$ is prime.}
\begin{tabular}{lcc | lcc} \label{T:CM}
        %\hline
	$x$
        & Actual
        & Expected
        & $x$
        & Actual
        & Expected
        \\ \hline \hline
20000000 &  49847 & 50063 &520000000 &  865909 & 866300   \\
40000000 &  91074 & 91134 &540000000 &  895323 & 895804   \\ 
60000000 &  129660 & 129648 &560000000 &  924773 & 925193    \\
80000000 &  166429 & 166631 &580000000 &  954215 & 954472   \\
100000000 &  202316 & 202534 &600000000 &  983415 & 983645   \\
120000000 &  237402 & 237612 &620000000 &  1012618 & 1012717   \\
140000000 &  271865 & 272024 &640000000 &  1041478 & 1041691   \\
160000000 &  305749 & 305882 &660000000 &  1070519 & 1070571   \\
180000000 &  338987 & 339266 &680000000 &  1099310 & 1099359  \\
200000000 &  372142 & 372237 &700000000 &  1127947 & 1128060   \\
220000000 &  404768 & 404844 &720000000 &  1156596 & 1156676  \\
240000000 &  437027 & 437124 &740000000 &  1185077 & 1185209 \\
260000000 &  469002 & 469110 &760000000 &  1213434 & 1213663  \\
280000000 &  500848 & 500827 &780000000 &  1241996 & 1242040  \\
300000000 &  532345 & 532298 &800000000 &  1270215 & 1270341  \\
320000000 &  563613 & 563542 &820000000 &  1298419 & 1298570  \\
340000000 &  594570 & 594575 &840000000 &  1326489 & 1326728  \\
360000000 &  625409 & 625412 &860000000 &  1354726 & 1354817  \\
380000000 &  656138 & 656065 &880000000 &  1382946 & 1382839  \\
400000000 &  686710 & 686546 &900000000 &  1410787 & 1410796  \\
420000000 &  716542 & 716864 &920000000 &  1438522 & 1438689  \\
440000000 &  746751 & 747028 &940000000 &  1466143 & 1466520  \\
460000000 &  776709 & 777047 &960000000 &  1493786 & 1494291  \\
480000000 &  806405 & 806928 &980000000 &  1521276 & 1522003  \\
500000000 &  836080 & 836677 &1000000000 &  1548766 & 1549657 
\end{tabular}
\end{table}

%************************************************************
\section{Example: $X_0(11)$} \label{S:X0(11)}

In this section we consider the elliptic curve $E=X_0(11)$ defined over $\QQ$.  The modular interpretation of $X_0(11)$ is not important for our purposes; it suffices to know that $y^2+y = x^3-x^2-10x-20$ is a minimal Weierstrass model for $E/\QQ$.  %E=ellinit([0,-1,1,-10,-20])
The curve $E$ has conductor $11$ and hence has good reduction away from $11$.
By Theorem \ref{T:Katz theorem}, $t_{E}$ divides $|E(\FF_p)|$ for each prime $p \nmid 2\cdot 11$; since $|E(\FF_3)|=5$, we deduce that $t_E$ divides $5$.  The rational point $(x,y)=(5,5)$ of $E$ has order $5$, and thus $5$ divides $t_{E}$.   We deduce that $t_E=5$ and in particular that $E(\QQ)_{\tors}$ is generated by $(5,5)$.  In this section we shall test Conjecture~\ref{C:main} with $t=t_E=5$.\\

  Lang and Trotter have worked out the Galois theory for this elliptic curve, and in particular have shown that Theorem~\ref{T:openimage} holds with $M=2\cdot 5\cdot 11$ (see \cite{Lang-Trotter}*{Part I, \S8} for full details). By Proposition~\ref{P:how to}, we have
\begin{align} \label{E:XO11 start}
\C_{E,5} &= \frac{\delta_{E,5}(2\cdot 5^2\cdot 11)}{\prod_{\ell|2\cdot5\cdot11}(1-1/\ell)} \prod_{\ell\nmid 2\cdot 5\cdot 11} \Big(1 - \frac{\ell^2-\ell-1}{(\ell-1)^3(\ell+1)}\Big)\\
& = \delta_{E,5}(2\cdot 5^2\cdot 11) \frac{345600}{78913} \prod_{\ell} \Big(1 - \frac{\ell^2-\ell-1}{(\ell-1)^3(\ell+1)}\Big). \notag
\end{align}
We shall now describe the structure of the group $G(2\cdot 5^2\cdot 11)$ and then compute $\delta_{E,5}(2\cdot 5^2\cdot 11)$.  Those not interested in this computation can skip ahead to the data.\\
 
For all $\ell \neq 5$, we have $G(\ell) = \Aut( E[\ell] )$.  There is a basis of $E[5^2]$ over $\ZZ/25\ZZ$ for which $G(5^2)$ becomes the group 
\[
\left\{ \left(\begin{matrix} 1+5a & 5b \\ 5c & u\end{matrix}\right): a,b,c\in \ZZ/25\ZZ, u \in (\ZZ/25\ZZ)^\times \right\}.
\]
To ease computation, identify $G(5^2)$ with this matrix group.  Fixing a basis, we can also identify $G(2)$ and $G(11)$ with the full groups $\GL(\ZZ/2\ZZ)$ and $\GL(\ZZ/11\ZZ)$ respectively.

Let $\varepsilon \colon G(2) \to \{\pm 1\}$ be the signature map (i.e, compose any isomorphism $G(2)\cong \mathfrak{S}_3$ with the usual signature), and define the homomorphisms 
\[
\phi_{11}\colon G(11)  \twoheadrightarrow \FF_{11}^\times/\{\pm 1\}, \quad
A \mapsto \pm\det(A)
\]
and
\[
\alpha \colon G(5^2) \twoheadrightarrow \ZZ/5\ZZ, \quad
\left(\begin{smallmatrix}1+5a & 5b \\ 5c & u\end{smallmatrix}\right)  \mapsto a \mod{5}.
\]
The group $\FF_{11}^\times/\{\pm 1\}$ is cyclic of order $5$ with generator $\pm2$, so it makes sense to define a homomorphism $\phi_5\colon G(5^2) \to \FF_{11}^\times/\{\pm 1\}$ by
\[
\phi_5(A) = (\pm 2)^{\alpha(A)}.
\]
We have a natural inclusion $G(2\cdot 5^2 \cdot 11) \subseteq G(2)\times G(5^2) \times G(11)$, which gives us the following description of $G(2\cdot 5^2\cdot 11)$:
\[
G(2\cdot 5^2 \cdot 11) = \left\{ (A_2,A_5,A_{11}) \in G(2)\times G(5^2) \times G(11): \phi_5(A_5)=\phi_{11}(A_{11}), \big(\tfrac{\det(A_{11})}{11}\big)=\varepsilon(A_2) \right\}.\] 
 
 \begin{lemma}
$|G(2\cdot 5^2 \cdot 11)|= 19800000.$
 \end{lemma}
 \begin{proof}
 We first use the fact that $\varepsilon$ surjects onto $\{\pm 1\}$, and $|G(2)|=6$.
\begin{align*} 
 |G(2\cdot 5^2 \cdot 11)|
 &= |\{ (A_2,A_5,A_{11}) \in G(2)\times G(5^2) \times G(11): \phi(A_5)=\phi(A_{11}), \big(\tfrac{\det(A_{11})}{11}\big)=\varepsilon(A_2) \}|\\
  &= 3\cdot |\{ (A_5,A_{11}) \in  G(5^2) \times G(11): \phi(A_5)=\phi(A_{11}), \big(\tfrac{\det(A_{11})}{11}\big)= 1 \}| \\
  & +   3\cdot |\{ (A_5,A_{11}) \in  G(5^2) \times G(11): \phi(A_5)=\phi(A_{11}), \big(\tfrac{\det(A_{11})}{11}\big)= -1 \}| \\
  &= 3\cdot |\{ (A_5,A_{11}) \in  G(5^2) \times G(11): \phi(A_5)=\phi(A_{11}) \}|
  \end{align*}
  We now use that $\phi_5$ and $\phi_{11}$ surject onto a common group of order 5.
\begin{align*} 
 |G(2\cdot 5^2 \cdot 11)| &= 3\cdot |\{ (A_5,A_{11}) \in  G(5^2) \times G(11): \phi(A_5)=\phi(A_{11}) \}| \\
  &= 3 |G(5^2)| |G(11)| /5 
  = 3  (5^3\cdot 20)  (11^2-1)(11^2-11) /5 =  19800000  \qedhere
\end{align*} 
 \end{proof}
 
 \begin{lemma}\label{L:X011 almost there}
 $\delta_{E,5}(2\cdot 5^2\cdot 11)= 9/50.$
 \end{lemma}
 \begin{proof}
 To ease notation, define $\B(m):=G(m)\cap \Psi_t(m)$.  First note that an element $A\in \GL(\ZZ/2\ZZ)=G(2)$ is in $\B(2)$ if and only if $\det(I-A)=\det(A) = 1$.  One quickly verifies that 
 $\B(2) = \left\{\left(\begin{smallmatrix} 1 & 1 \\1 & 0\end{smallmatrix}\right),\left(\begin{smallmatrix} 0 & 1 \\1 & 1\end{smallmatrix}\right)\right\}.$   
 These two elements have order three, so $\varepsilon(A)=1$ for all $A\in \B(2)$.
\begin{align*} 
 |\B(2\cdot 5^2 \cdot 11)| 
 &= | \{ (A_2,A_5,A_{11})\in \B(2)\times \B(5^2)\times \B(11): \phi(A_5)=\phi(A_{11}), \big(\tfrac{\det(A_{11})}{11}\big)=\varepsilon(A_2) \}| \\
 &=  2\cdot | \{ (A_5,A_{11})\in \B(5^2)\times \B(11): \phi(A_5)=\phi(A_{11}), \det(A_{11}) \in (\FF^\times_{11})^2 \}|\\
 &= 2\sum_{ x \in \FF_{11}^\times/\{\pm 1\}} | \{A\in \B(5^2): \phi_5(A) = x \}| \cdot |\{A\in \B(11): \phi_{11}(A)=x, \det(A)\in (\FF_{11}^\times)^2 \}|
\end{align*}

 Take any $A=\left(\begin{smallmatrix}1+5a & 5b \\5c & u\end{smallmatrix}\right)\in G(5)$.  We have $ \det(I-A) = 5a(u-1) \in t_E (\ZZ/25\ZZ)^\times= 5(\ZZ/25\ZZ)^\times$ if and only if $a\not \equiv 0 \pmod{5}$ and $u \not \equiv 1 \pmod{5}$.  
Given $a\in\ZZ/5\ZZ$, we find that
\[
|\{ A \in \B(5^2): \alpha(A)=a \}| = \begin{cases}
    5^2\cdot15=375  & \text{ if }a\not \equiv 0 \pmod{5} \\
     0 & \text{ if }a\equiv 0 \pmod{5},
\end{cases}
\]
and hence for $b\in \FF^\times_{11}$,
\begin{align*} 
|\{A\in \B(5^2): \phi_{5}(A)=\pm b \} |
 &= \begin{cases}
     375  & \text{ if }b\neq \pm 1 \\
     0 & \text{ if } b=\pm 1.
\end{cases}
\end{align*}
Our expression for $ |\B(2\cdot 5^2 \cdot 11)| $ thus simplifies to the following,
\[ 
|\B(2\cdot 5^2 \cdot 11)| 
=750 \sum_{ x \in \FF_{11}^\times/\{\pm 1\} -\{ \{\pm 1\} \}} |\{A\in \B(11): \phi_{11}(A)=x, \;  \det(A)\in (\FF_{11}^\times)^2 \}| .
\]
Take any $x\in \FF_{11}^\times/\{\pm 1 \}$.  Since $-1$ is not a square in $\FF_{11}^\times$, the class $x$ contains a unique element $b_x\in (\FF_{11}^\times)^2$.
\begin{align*} 
& | \{ A \in \B(11): \phi_{11}(A)=x,\; \det(A)\in (\FF_{11}^\times)^2 \}| = | \{ A \in \B(11): \det(A)= b_x \}|.
\end{align*}
So our expression for $ |\B(2\cdot 5^2 \cdot 11)| $ simplifies further to
\[ 
|\B(2\cdot 5^2 \cdot 11)| 
=750 \sum_{ b\in (\FF_{11}^\times)^2 - \{1\} } |\{A\in \GL(\FF_{11}): \det(A)=b,\; \det(I-A) \neq 0\}|.
\]
Using Lemma~\ref{L:eigenvalues}, we obtain
\begin{align*} 
|\B(2\cdot 5^2 \cdot 11)| &= 
750 \sum_{ b\in (\FF_{11}^\times)^2 - \{1\} } \Bigl( |\GL(\FF_{11})|/10 -  (11^2+11) \Bigr)\\
&= 750 \cdot 4 \cdot ((11^2-1)(11^2-11)/10 - (11^2+11)) = 3564000. 
\end{align*} 
Therefore using the previous lemma, we have
\[
\delta_{E,5}(2\cdot 5^2\cdot 11) = |\B(2\cdot 5^2 \cdot 11)|/|G(2\cdot 5^2 \cdot 11)|=3564000/19800000=9/50.			\qedhere
\]
\end{proof}

We finally describe the constant $\C_{X_0(11),5}$ from Conjecture~\ref{C:main}.  Lemma~\ref{L:X011 almost there} and (\ref{E:XO11 start}) imply that
\begin{equation} \label{E:X011 constant}
\C_{X_0(11),5} = \frac{62208}{78913}  \prod_{\ell} \Bigl(1 - \frac{\ell^2-\ell - 1}{(\ell-1)^3(\ell+1)}  \Bigr).  
\end{equation}

In the following table, the ``expected number'' of $p\leq x$ with $p\neq 11$ such that $|X_0(11)(\FF_p)|/5$ is prime is
\begin{equation} \label{E:X011 prediction}
\C_{X_0(11),5}\int_{6}^x \frac{1}{\log(u+1)-\log 5 }\frac{du}{\log u}
\end{equation}
rounded to the nearest integer, where $\C_{X_0(11),5}$ is estimated using (\ref{E:X011 constant}) and (\ref{E:numerical constant}).

\begin{table}[H]
\caption{Number of $p\leq x$ with $p\neq 11$ such that $|X_0(11)(\FF_p)|/5$ is prime.}
\begin{tabular}{lcc | lcc} \label{T:X0(11)}
        %\hline
	$x$
        & Actual
        & Expected
        & $x$
        & Actual
        & Expected
        \\ \hline \hline
20000000    &36051 & 36091 & 520000000 &629151& 628797 \\
40000000    &66143 & 65814 & 540000000 &650676& 650253\\
60000000    &94050 & 93715 & 560000000 &671998& 671626\\
80000000    &120806 & 120523 & 580000000 &693377& 692921\\
100000000   &146748 & 146560 & 600000000 &714783& 714139\\
120000000   &172172 & 172007 & 620000000 &735972& 735285\\
140000000   &197180 & 196979 & 640000000 &756879& 756360\\
160000000   &221586 & 221554 & 660000000 &777830& 777368\\
180000000   &245768 & 245790 & 680000000 &798736& 798311\\
200000000   &269776 & 269730 & 700000000 &819665& 819190\\
220000000   &293290 & 293410 & 720000000 &840621& 840008\\
240000000   &316771 & 316855 & 740000000 &861196& 860768\\
260000000   &340034 & 340090 & 760000000 &881992& 881470\\
280000000   &363448 & 363133 & 780000000 &902549& 902117\\
300000000   &386413 & 385999 & 800000000 &923181& 922709\\
320000000   &409103 & 408703 & 820000000 &943660& 943250\\
340000000   &431644 & 431255 & 840000000 &964135& 963740\\
360000000   &453854 & 453667 & 860000000 &984561& 984180\\
380000000   &476378 & 475947 & 880000000 &1005037& 1004572\\
400000000   &498621 & 498103 & 900000000 &1025528& 1024917\\
420000000   &520651 & 520143 & 920000000 &1045814& 1045217\\
440000000   &542604 & 542072 & 940000000 &1066059& 1065472\\
460000000   &564364 & 563898 & 960000000 &1086151& 1085683\\
480000000   &586046 & 585624 & 980000000 &1106398& 1105852\\
500000000   &607563 & 607255 & 1000000000 &1126420& 1125980\\
 \hline
\end{tabular}
\end{table}

\begin{remark}
The term $\log 5$ in (\ref{E:X011 prediction}) is numerically important.  For example, we find that $\C_{X_0(11),5} \cdot \int_{6}^{10^9} (\log(u+1)\log u)^{-1} du \approx 1033120$, which is a worse approximation of $P_{X_0(11),5}(10^9)$ than that given in Table~\ref{T:X0(11)}.
\end{remark}

\section{Recent progress}

We briefly describe some of the progress that has been made on Koblitz's conjecture.  This very short survey is not meant to be exhaustive and sometimes we only state special cases of results; one should consult the cited papers for more details and developments.   In this section, we limit ourselves to elliptic curves defined over $\QQ$.\\

First of all, there are currently no examples where Conjecture~\ref{C:main} is known to hold besides those trivial cases where $C_{E,t}=0$ (and thus have a congruence obstruction).   Moreover, there are no known examples of elliptic curves $E$ and integers $t$ for which $\lim_{x\to \infty} P_{E,t}(x)=\infty$.\\

 Much of the recent progress has been made by applying methods from sieve theory (including methods that were used to study twin primes or Sophie Germain primes).     Recall that the conjecture that there are infinitely many Sophie Germain primes is equivalent to there being infinitely many primes $p$ for which $(p-1)/2$ is prime (this is an analogue of Conjecture~\ref{C:main} with $K=\QQ$, $t=2$ and $E$ replaced by the group scheme $\GG_m$).

\subsection{Non-CM curves}
Let $E$ be an elliptic curve over $\QQ$ without complex multiplication.  Miri and Murty \cite{Miri-Murty} showed, assuming GRH, that there are $\gg x/(\log x)^2$ primes $p\leq x$ for which $|E(\FF_p)|$ has at most $16$ prime divisors.  Steuding and Weng \cite{Steuding-Weng,Steuding-Weng-E} improved this to $9$ factors.    Assuming GRH and $t_E=1$, David and Wu \cite{David-Wu} have shown that
\[
|\{p\leq x : |E(\FF_p)| \text{ has at most $8$ prime factors}\}| \geq 2.646\cdot \C_{E,1} \frac{x}{(\log x)^2}
\]
for $x\gg_E 1$, where $\C_{E,1}$ is the constant of Conjecture~\ref{C:main}.\\

We now mention some upper bounds obtained under GRH (though weakening of this conjecture can also be used).   Cojocaru \cite{Cojocaru-Koblitz} proved that $P_{E,1}(x) \ll x/(\log x)^2;$ which of course should be the best possible general bound, up to improvement of the implicit constant. David and Wu \cite{David-Wu} have shown that for any $\varepsilon>0$, one has 
\[
P_{E,1} \leq (10+\varepsilon)\C_{E,1}\frac{x}{(\log x)^2}
\]
 for all $x\gg_{E,\varepsilon} 1$.   In the general setting of Conjecture~\ref{C:main}, one has the bound
 \[
 P_{E,t}(x) \leq (22+o(1)) \C_{E,t} \frac{x}{(\log x)^2}
 \]
where the $o(1)$ term depends on $E$ and $t$; this is Theorem~1.3 of \cite{Zywina-Large} (this theorem uses $t=t_E$, but the proof carries through for general $t$).   

For \emph{unconditional} upper bounds, Cojocaru \cite{Cojocaru-Koblitz} proved that $P_{E,1}(x) \ll x/(\log x \log\log \log x)$.  This can be strengthened to $P_{E,1}(x)\leq (24+ o(1)) \C_{E,1} \cdot x/(\log x \log\log x),$ see \cite{Zywina-Large}*{Theorem~1.3}.  However, this is still not strong enough to prove that 
\[
\sum_{p,\, |E(\FF_p)| \text{ is prime}} \frac{1}{p}  < \infty
\] 
(which would be the analogue of Brun's theorem that $\sum_{p\text{ and } p+2 \text{ are prime}} 1/p < \infty$).

\subsection{CM elliptic curves}
Now consider a CM elliptic curve $E$ over $\QQ$.   If $t_E=1$, then Cojocaru \cite{Cojocaru-Koblitz}*{Theorem 4} has shown that
\[
|\{ p \leq x : |E(\FF_p)| \text{ has at most $5$ prime factors}\}| \gg \frac{x}{(\log x)^2}
\]
(note that this theorem does \emph{not} depend on GRH).  If $E$ has CM by the maximal order $\OO_F$ of an imaginary quadratic extension $F/\QQ$, then Jim\'enez Urroz \cite{Urroz} has proved that
\[
|\{ p \leq x:  p \text{ splits in $F$,\,  $|E(\FF_p)|/t_{E_F}$ has at most $2$ prime factors}\}| \gg \frac{x}{(\log x)^2}
\]
(this extends a result of Iwaniec and Jim\'enez Urroz mentioned at the beginning of \S\ref{S:CM example}).

\subsection{The conjecture on average}

We now consider the functions $P_{E,1}(x)$ averaged over a family of elliptic curves.    Fix $\alpha> 1/2$ and $\beta>1/2$ with $\alpha+\beta>3/2$.  Let $\calF(x)$ be the set of $(a,b)\in \ZZ^2$ with $|a|\leq x^\alpha$ and $|b|\leq x^\beta$ for which $4a^3+27b^2 \neq 0$.   For $(a,b)\in \calF(x)$, let $E(a,b)$ be the elliptic curve over $\QQ$ defined by the affine equation $Y^2=X^3+aX+b$.   Balog, Cojocaru, and David \cite{Balog-Cojocaru-David} have proved that
\begin{equation} \label{E:onaverage}
\frac{1}{|\calF(x)|} \sum_{(a,b) \in\calF(x)} P_{E(a,b),1}(x) \sim \mathfrak{C} \frac{x}{(\log x)^2}
\end{equation}
as $x\to \infty$, where $\displaystyle\mathfrak{C}=\prod_\ell  \Bigl(1- \frac{\ell^2-\ell-1}{(\ell-1)^3(\ell+1)} \Bigr).$    Informally, this says that Koblitz's  conjecture holds ``on average''.  

\subsection{The constant on average}
Assuming a positive answer to a question of Serre\footnote{Does there exists a  constant $C$ such that for any non-CM elliptic curve $E/\QQ$, we have $\rho_{\ell}(\Gal(\Qbar/\QQ))=\GL(\ZZ/\ell\ZZ)$ for all primes $\ell\geq C$?}, Jones \cite{Jones-Ann} proved that (\ref{E:onaverage}) is also true on the level of constants; i.e., 
\[
\lim_{x\to \infty} \frac{1}{|\calF(x)|} \sum_{(a,b) \in\calF(x)} \C_{E(a,b),1} = \mathfrak{C}.
\]
%	\commentm{
%	This can be proven unconditionally, and will be the topic of future work.
%}
Finally, we explain how this can be proven \emph{unconditionally} (we state it in a fashion similar to \cite{Jones-Ann}*{Theorem~6}).   
\begin{prop}
Let $\calF(x)$ be the set of $(a,b)\in \ZZ^2$ with $|a|\leq x$ and $|b|\leq x$ such that $4a^3+27b^2\neq 0$.  Then there is an absolute constant $\gamma>0$ such that for any integer $k\geq 1$, we have
\[
\frac{1}{|\calF(x)|} \sum_{(a,b) \in\calF(x)} |\C_{E(a,b),1} - \mathfrak{C}|^k \ll_k \frac{(\log x)^\gamma}{\sqrt{x}}.
\]
\end{prop}
\begin{proof}(Sketch)
We first consider a fixed non-CM elliptic curve $E$ over $\QQ$.   Let $M$ be the positive squarefree integer for which $\ell\nmid M$ if and only if $\ell \geq 5$ and $G(\ell) \supseteq \SL(\ZZ/\ell\ZZ)$.     %Let $m$ be a positive squarefree integer relatively prime to $M$.  
Some group theory shows that
\[
G(Mm) = G(M) \times \prod_{\ell|m} \Aut(E[\ell])
\]
for any squarefree integer $m$ relatively prime to $M$.   
%(use Goursat's lemma and the fact that the groups $\SL(\ZZ/\ell\ZZ)$ have no common non-trivial quotients...
By \cite{Masser}*{Theorem~3}, there is an absolute constant $\kappa\geq 0$ such that
\[
M\ll \max\{ 1, h(E)^\kappa \}
\]
where $h(E)$ is the logarithmic absolute semistable Faltings height of $E$.      By \cite{Silverman-heights},  we have $h(E)\ll h(j_E)$ where $j_E$ is the $j$-invariant of $E$ and $h$ is the usual height of a rational number.

By Proposition~\ref{P:how to} with the above $M$, 
\[
 \C_{E,1} = \dfrac{ \delta_{E,1}\big(\prod_{\ell|M}\ell\big)}{\prod_{\ell|M} (1-1/\ell)} \prod_{\ell \nmid M} \Big(1 - \frac{\ell^2-\ell-1}{(\ell-1)^3(\ell+1)}\Big) \leq \prod_{\ell|M} (1-1/\ell)^{-1}.
\]
If $y$ is the smallest prime for which $\prod_{\ell\leq y} \ell \leq M$, then 
\[
\prod_{\ell|M} (1-1/\ell)^{-1}\leq \prod_{\ell\leq y} (1-1/\ell)^{-1} \ll \log y
\]
where the last inequality follows from Mertens' theorem.  So
\[
\C_{E,1} \ll  \log y  \ll \log\big(\sum_{\ell\leq y} \log \ell \big) \leq \log \log M.
\]
Therefore, for any non-CM elliptic curve over $\QQ$ we have
\begin{equation} \label{E:constant bound}
\C_{E,1} \ll \log\log\log(h(j_E)+16)
\end{equation}
(the $16$ is added simply to make sure the right-hand side is alway well-defined and positive).   One can also check that $\C_{E,1} \ll 1$ for CM elliptic curves $E/\QQ$.\\

Let $\mathcal{S}(x)$ be the set of $(a,b)\in \calF(x)$ for which $E(a,b)$ is a Serre curve (cf.~\S\ref{S:Serre curve section}).   Theorem~10 of \cite{Jones-Ann} implies that 
\begin{equation} \label{E:almost done 1}
\frac{1}{|\calF(x)|} \sum_{(a,b) \in\mathcal{S}(x)} |\C_{E(a,b),1} - \mathfrak{C}|^k \ll_k  \frac{(\log x)^8}{\sqrt{x}};
\end{equation}
a key point is that the difference $\C_{E(a,b),1} - \mathfrak{C}$ has a nice description (cf.~Proposition~\ref{P:Serre curves constant}).

For each $(a,b)\in \calF(x)$, we have $h(j_{E(a,b)})\ll \log x$.   Therefore by (\ref{E:constant bound}) we have
\begin{align*}
\frac{1}{|\calF(x)|} \sum_{(a,b) \in\calF(x)-\mathcal{S}(x)} |\C_{E(a,b),1} - \mathfrak{C}|^k & \ll_k  \frac{|\calF(x)-\mathcal{S}(x)|}{|\calF(x)|} (\log\log\log\log x)^{k}.
\end{align*}
By \cite{Jones-AAECASCF}*{Theorem~4}, there is a constant $\beta>0$ such that  $|\calF(x)-\mathcal{S}(x)|/|\calF(x)| \ll (\log x)^\beta/\sqrt{x}$.  Therefore
\begin{align}\label{E:almost done 2}
\frac{1}{|\calF(x)|} \sum_{(a,b) \in\calF(x)-\mathcal{S}(x)} |\C_{E(a,b),1} - \mathfrak{C}|^k & \ll_k  \frac{(\log x)^\beta}{\sqrt{x}}
\end{align}
for some constant $\beta>0$.  The proposition follows immediately by combining (\ref{E:almost done 1}) and (\ref{E:almost done 2}).
\end{proof}


\begin{bibdiv}
\begin{biblist}

\bib{Magma}{article}{
author={Bosma, W.}, author={Cannon, J.}, author={Playoust, C.}, 
title={The {M}agma algebra system. {I}. {T}he user language}, 
journal={J. Symbolic Comput.}, 
volume={24},
year={1997},
pages={235--265}
}

\bib{Cojocaru-Koblitz}{article}{
   author={Cojocaru, Alina Carmen},
   title={Reductions of an elliptic curve with almost prime orders},
   journal={Acta Arith.},
   volume={119},
   date={2005},
   number={3},
   pages={265--289},
%   issn={0065-1036},
%   review={\MR{2167436 (2006f:11111)}},
}

\bib{Balog-Cojocaru-David}{article}{
author={Balog, Antal}, author={Cojocaru, Alina}, author={David, Chantal},
title={Average twin prime conjecture for elliptic curves},
note={\href{http://arxiv.org/abs/0709.1461}{arXiv:0709.1461v1} [math.NT]},
date={2007}
}

\bib{David-Wu}{article}{
author={David, Chantal},
author={Wu, Jie},
title={Almost prime values of the order of elliptic curves over finite fields},
note={\href{http://arxiv.org/abs/0812.2860}{arXiv:0812.2860v1} [math.NT]},
date={2008}
}


\bib{Aaron}{thesis}{
  author={Greicius, Aaron},
  title={Elliptic curves with surjective global Galois representation},
  school = {University of California, Berkeley},
  type = {Ph.D. thesis},
  year = {2007},
  }
   
   \bib{Hardy-Littlewood}{article}{
   author={Hardy, G.H.},
   author={Littlewood, J.E.},
   title={Some problems of `Partitio numerorum': III: on the expression of a number as a sum of primes},
   journal={Acta Math.},
   volume={44},
   date={1923},
   pages={1\ndash 70}
   }
   
   
\bib{Iwaniec}{article}{
author={Iwaniec, H.}, 
author={Jim\'enez Urroz, J.},
title={Orders of CM elliptic curves modulo p with at most two primes}, note={\href{http://upcommons.upc.edu/e-prints/handle/2117/1169}{http://upcommons.upc.edu/e-prints/handle/2117/1169}}
date={2006}
}

\bib{Urroz}{article}{
   author={Jim{\'e}nez Urroz, Jorge},
   title={Almost prime orders of CM elliptic curves modulo $p$},
   conference={
      title={Algorithmic number theory},
   },
   book={
      series={Lecture Notes in Comput. Sci.},
      volume={5011},
      publisher={Springer},
      place={Berlin},
   },
   date={2008},
   pages={74--87},
   %review={\MR{2467838}},
}



\bib{Jones-AAECASCF}{article}{
   author={Jones, Nathan},
   title={Almost all elliptic curves are Serre curves}, 
   %note={\href{http://arxiv.org/abs/math/0611096}{arXiv:math/0611096v1} [math.NT],  
   journal={Transactions of the AMS (to appear)},
   date={2009}
}


\bib{Jones-Ann}{article}{
   author={Jones, Nathan},
   title={Averages of elliptic curve constants}, 
   %note={\href{http://arxiv.org/abs/0711.3484v1}{arXiv:math/0711.3484v1} [math.NT]},
   journal = {Math.~Ann. (to appear)},
   date={2009}
}

   \bib{KatzInv}{article}{
   author={Katz, Nicholas M.},
   title={Galois properties of torsion points on abelian varieties},
   journal={Invent. Math.},
   volume={62},
   date={1981},
   number={3},
   pages={481--502},
%   issn={0020-9910},
%   review={\MR{604840 (82d:14025)}},
}

\bib{KatzSarnak}{book}{
   author={Katz, Nicholas M.},
   author={Sarnak, Peter},
   title={Random matrices, Frobenius eigenvalues, and monodromy},
   series={American Mathematical Society Colloquium Publications},
   volume={45},
   publisher={American Mathematical Society},
   place={Providence, RI},
   date={1999},
   pages={xii+419},
%   isbn={0-8218-1017-0},
%   review={\MR{1659828 (2000b:11070)}},
}

\bib{Koblitz}{article}{
   author={Koblitz, Neal},
   title={Primality of the number of points on an elliptic curve over a
   finite field},
   journal={Pacific J. Math.},
   volume={131},
   date={1988},
   number={1},
   pages={157--165},
%   issn={0030-8730},
%   review={\MR{917870 (89h:11023)}},
}
  
\bib{LangAlgebra}{book}{
   author={Lang, Serge},
   title={Algebra},
   series={Graduate Texts in Mathematics},
   volume={211},
   edition={3},
   publisher={Springer-Verlag},
   place={New York},
   date={2002},
   pages={xvi+914},
%   isbn={0-387-95385-X},
%   review={\MR{1878556 (2003e:00003)}},
}
\bib{Lang-Trotter}{book}{
   author={Lang, Serge},
   author={Trotter, Hale},
   title={Frobenius distributions in ${\rm GL}\sb{2}$-extensions},
   note={Distribution of Frobenius automorphisms in ${\rm
   GL}\sb{2}$-extensions of the rational numbers;
   Lecture Notes in Mathematics, Vol. 504},
   publisher={Springer-Verlag},
   place={Berlin},
   date={1976},
   pages={iii+274},
%   review={\MR{0568299 (58 \#27900)}},
}

\bib{Masser}{article}{
   author={Masser, David},
   title={Multiplicative isogeny estimates},
   journal={J. Austral. Math. Soc. Ser. A},
   volume={64},
   date={1998},
   number={2},
   pages={178--194},
   issn={0263-6115},
%   review={\MR{1619802 (2000a:11089)}},
}

\bib{Miri-Murty}{article}{
   author={Miri, S. Ali},
   author={Murty, V. Kumar},
   title={An application of sieve methods to elliptic curves},
   conference={
      title={Progress in cryptology---INDOCRYPT 2001 (Chennai)},
   },
   book={
      series={Lecture Notes in Comput. Sci.},
      volume={2247},
      publisher={Springer},
      place={Berlin},
   },
   date={2001},
   pages={91--98},
%   review={\MR{1934487 (2003i:11137)}},
}


\bib{PARI}{manual}{
%@manual{PARI2,
      author={PARI~Group, The},
      organization = {The PARI~Group},
      title       = {PARI/GP, version {\tt 2.3.4}},  
      year         = {2008},
      address      = {Bordeaux},
      note         = {available from \url{http://pari.math.u-bordeaux.fr/}}
    }
    
    
    
\bib{SerreInv}{article}{
   author={Serre, Jean-Pierre},
   title={Propri\'et\'es galoisiennes des points d'ordre fini des courbes
   elliptiques},
  % language={French},
   journal={Invent. Math.},
   volume={15},
   date={1972},
   number={4},
   pages={259--331},
%   issn={0020-9910},
%   review={\MR{0387283 (52 \#8126)}},
}

\bib{SerreCheb}{article}{
   author={Serre, Jean-Pierre},
   title={Quelques applications du th\'eor\`eme de densit\'e de Chebotarev},
  % language={French},
   journal={Inst. Hautes \'Etudes Sci. Publ. Math.},
   number={54},
   date={1981},
   pages={323--401},
%   issn={0073-8301},
%   review={\MR{644559 (83k:12011)}},
}


\bib{SerreTate:GoodReduction}{article}{
   author={Serre, Jean-Pierre},
   author={Tate, John},
   title={Good reduction of abelian varieties},
   journal={Ann. of Math. (2)},
   volume={88},
   date={1968},
   pages={492--517},
%   issn={0003-486X},
%   review={\MR{0236190 (38 \#4488)}},
}

\bib{Silverman-heights}{article}{
   author={Silverman, Joseph H.},
   title={Heights and elliptic curves},
   conference={
      title={Arithmetic geometry},
      address={Storrs, Conn.},
      date={1984},
   },
   book={
      publisher={Springer},
      place={New York},
   },
   date={1986},
   pages={253--265},
%   review={\MR{861979}},
}

\bib{Steuding-Weng}{article}{
   author={Steuding, J{\"o}rn},
   author={Weng, Annegret},
   title={On the number of prime divisors of the order of elliptic curves
   modulo $p$},
   journal={Acta Arith.},
   volume={117},
   date={2005},
   number={4},
   pages={341--352},
   %issn={0065-1036},
   %review={\MR{2140162 (2005m:11178)}},
}

\bib{Steuding-Weng-E}{article}{
   author={Steuding, J{\"o}rn},
   author={Weng, Annegret},
   title={Erratum: ``On the number of prime divisors of the order of
   elliptic curves modulo $p$'' [Acta Arith. {\bf 117} (2005), no. 4,
   341--352; MR 2140162]},
   journal={Acta Arith.},
   volume={119},
   date={2005},
   number={4},
   pages={407--408},
   %issn={0065-1036},
   %review={\MR{2189069 (2006g:11195)}},
}

\bib{Stevenhagen}{article}{
   author={Stevenhagen, Peter},
   title={The correction factor in Artin's primitive root conjecture},
%   language={English, with English and French summaries},
   note={Les XXII\`emes Journ\'ees Arithmetiques (Lille, 2001)},
   journal={J. Th\'eor. Nombres Bordeaux},
   volume={15},
   date={2003},
   number={1},
   pages={383--391},
%   issn={1246-7405},
%   review={\MR{2019022 (2004j:11141)}},
}

\bib{Zywina-maximalgalois}{article}{
   author={Zywina, David},
   title={Elliptic curves with maximal Galois action on their torsion points}, 
   note={\href{http://arxiv.org/abs/0809.3482}{arXiv:0809.3482v1} [math.NT]},
   date={2008}
}

\bib{Zywina-Large}{article}{
   author={Zywina, David},
   title={The Large Sieve and Galois Representations}, 
   note={\href{http://arxiv.org/abs/0812.2222}{arXiv:0812.2222v1} [math.NT]},
   date={2008}
}


\end{biblist}
\end{bibdiv}
\end{document}